[1994/12/01]% LaTeX date must December 1994 or later
\documentclass[11pt,reqno]{article}

\title{Random Walks on Directed Covers of Graphs}

\author{Lorenz A. Gilch\footnote{email: gilch@TUGraz.at; supported by German Research Foundation (DFG) grant
GI 746/1-1} , Sebastian M\"uller\footnote{email:
dr.sebastian.mueller@gmail.com; supported by German Research Foundation (DFG)
grant MU 2868/1-1 }\vspace{0.5cm}\\ Institut f\"ur mathematische
Strukturtheorie \\
  Graz University of Technology\\ Steyrergasse 30/III\\ A-8010 Graz\\ Austria}

\usepackage{amsmath,amsthm,amsfonts}
\usepackage{t1enc}
\usepackage[latin1]{inputenc}
\usepackage[ps2pdf]{hyperref}
\usepackage{graphicx}
\usepackage{color}
\usepackage{umoline}

\chardef\bslash=`\\ % p. 424, TeXbook
\newtheorem{thm}{Theorem}[section]

\newtheorem{cor}[thm]{Corollary}
\newtheorem{lem}[thm]{Lemma}
\newtheorem{prop}[thm]{Proposition}

\theoremstyle{definition}

\theoremstyle{remark}
\newtheorem{rem}{Remark}[section]

\newtheorem{ex}{Example}[section]

\def\orig{o}
\newcommand{\eps}{\varepsilon}

\newcommand{\R}{\mathbb{R}}
\newcommand{\C}{\mathbb{C}}
\newcommand{\Z}{\mathbb{Z}}
\newcommand{\N}{\mathbb{N}}
\newcommand{\T}{\mathcal{T}}
\newcommand{\U}{\mathcal{U}}
\renewcommand{\P}{\mathbb{P}}
\newcommand{\E}{\mathbb{E}}
\newcommand{\1}{\textbf{1}}
\newcommand{\p}{\textbf{p}}

\newcommand{\m}{\textbf{m}}
\newcommand{\e}[1]{\mathbf{e}_{#1}}
\newcommand{\W}[1]{\mathbf{W}_{#1}}

\DeclareMathOperator{\esssup}{ess\,sup}

\newcommand{\eval}[2][\right]{\relax
  \ifx#1\right\relax \left.\fi#2#1\rvert}

\begin{document}
\maketitle

{\abstract Directed covers of finite graphs are also known as
periodic trees or trees with finitely many cone types. We expand
the existing theory of directed covers of finite graphs to those
of infinite graphs. While the lower growth rate still equals the
branching number, upper and lower growth rates do not longer
coincide in general. Furthermore, the behaviour of random walks on
directed covers of infinite graphs is more subtle. We provide a
classification in terms of recurrence and transience and point out that the
critical random walk may be recurrent or transient. Our proof is
based on the observation that recurrence of the random walk is
equivalent to the almost sure extinction of an appropriate
branching process. Two examples in random environment are
provided: homesick random walk on infinite percolation clusters
and random walk in random environment on directed covers. Furthermore, we
calculate, under
reasonable assumptions, the rate of escape with respect
to suitable length functions and prove the existence of the
asymptotic entropy providing an explicit formula which is also a
new result for directed covers of finite graphs. In particular,
the asymptotic entropy of random walks on directed covers of
finite graphs is positive if and only if the random walk is
transient.}
\newline {\scshape Keywords:} trees,  random walk, recurrence, transience,
upper  Collatz-Wielandt number, branching process, rate of
escape, asymptotic entropy
\newline {\scshape AMS 2000 Mathematics Subject Classification: } 05C05, 60J10 , 60F05,  60J85

\renewcommand{\sectionmark}[1]{}

\section{Introduction}
Suppose we are given a connected, directed graph $G$ with vertex set $V$,
edge set $E$, and root $i_0$. We
construct  a labelled tree $\mathcal{T}$ from $G$. We start with
the root that is labelled with $i_0$. Recursively, if $x$ is a
vertex in the tree with label $i\in V$, then $x$ has $d(i,j)$ successors
with label $j$ if and only if there are $d(i,j)$ edges from $i$ to $j$ in
$G$. The tree $\mathcal{T}$ is called the \emph{directed cover} of
$G$. This model generalizes previously investigated ones, namely
random walks on  directed covers of finite graphs.  These trees
are also known as \emph{periodic trees}, compare with
\textsc{Lyons} \cite{lyons90}, or \emph{trees with finitely many
cone types}, compare with \textsc{Nagnibeda and Woess}
\cite{nagnibeda:02}.

If $G$ is finite the directed cover has the  property that the
growth rate exists and equals the branching number of $\T$, compare with
\textsc{Lyons} \cite{lyons:book}. \textsc{Lyons} \cite{lyons90}
investigated the behaviour of homesick random walks on those
trees. \textsc{Takacs} \cite{takacs} computed a formula for the
rate of escape of homesick random walks. \textsc{Nagnibeda and
Woess} \cite{nagnibeda:02} studied more general random walks on
these trees and gave a criterion for transience, null-recurrence,
and positive recurrence. This criterion depends on the largest
eigenvalue of a positive matrix arising from the transition
probabilities. They also computed, among other things, a formula for
the rate of escape of random walks on directed covers of finite
graphs. It is worth mentioning that the model of random strings
discussed in \textsc{Gairat et al.} \cite{gairat} offers a
different point of view of this model.

Most of the arguments in the finite case are based on the existence of
the Perron-Frobenius eigenvalue of appropriate non-negative
matrices. In the infinite setting, the matrices become
non-negative operators and the existence of a largest eigenvalue
can no longer be guaranteed. If there exists a largest eigenvalue
with positive left and right eigenvectors the study is analogous to
the finite case, but in general the behaviour becomes more
subtle. In particular, the lower and upper growth rates of the
cover, that are defined by $\liminf_{n\to\infty}
|\mathcal{T}_n|^{1/n}$ and $ \limsup_{n\to\infty}
|\mathcal{T}_n|^{1/n}$, where $\mathcal{T}_n$ is the number of
vertices in $\mathcal{T}$ at height $n$,  are no longer equal,
compare with Example \ref{ex:no_gr}. However,  the lower growth
rate and branching number coincide and equal the upper
Collatz-Wielandt number of the adjacency matrix of $G$, see
Theorem \ref{thm:growth}.

The first main result, Theorem
\ref{thm:rec}, is the classification of random walks on
directed covers according to their transience and recurrence behaviour. This result is given in terms of a bounded operator
$M$ that describes the relation between  \emph{forward} and
\emph{backward} probabilities of the random walk.  We show that if
the upper Collatz-Wielandt number, $\lambda^+(M)$, of this
operator is smaller than $1$ the process is recurrent and if it is
greater than $1$ it is transient. In the critical case,
$\lambda^+(M)=1$, the random walk may be recurrent or transient,
compare with Subsection \ref{subseq:ex} and Example
\ref{ex:lambda=1}. This is in contrast with the finite setting,
where the critical random walk is recurrent; e.g. see
\textsc{Nagnibeda and Woess} \cite{nagnibeda:02}. The idea of our
proof is based on the observation that recurrence of the random
walk is equivalent to a.s. global extinction of an appropriate
infinite-type Galton-Watson process, compare with Theorem
\ref{thm:br_mc}. Therefore, our approach gives an alternative
proof, that uses standard results for multi-type Galton-Watson
processes, for random walks on directed covers of finite graphs.
In Subsection \ref{subseq:ex} we study several examples where a
complete classification is given.

Another result is that the rate of escape w.r.t. different
appropriate length functions exists under reasonable assumptions. For this
purpose, we assume the
spectral radius to be strictly smaller than $1$ and 
positive recurrence of a new Markov chain on $G$, which describes the end of
the tree to which the random walk on $\T$ converges. As in \textsc{Nagnibeda
  and Woess} \cite{nagnibeda:02}, the existence of the rate of escape can not
be shown by straight-forward arguments using Kingman's subadditive ergodic
theorem. 
 Thus, we use modified exit times from \cite{nagnibeda:02} to prove 
 existence of the rate of escape with respect to different length functions, see
Theorem \ref{thm:rateofescape}. This enables us to prove 
another main result of the paper, namely the existence of
the asymptotic entropy $h=\lim_{n\to\infty}
\E\bigl[-\frac{1}{n}\log \pi_n(X_n)\bigr]$ under reasonable
assumptions, where $(X_n)_{n\in\N_0}$ is a random walk on a
directed cover and $\pi_n$ is its distribution at time $n$. This
result, whose proof envolves generating functions techniques, is also new for
the case when $G$ is finite. While it
is well-known that entropy (introduced by \textsc{Avez}
\cite{avez72}) exists for random walks on groups existence for
random walks on other structures is not known a priori. For more details about
entropy of random walks on groups see \textsc{Kaimanovich and Vershik}
\cite{kaimanovich-vershik} and \textsc{Derriennic}
\cite{derrienic}. 
In particular, we show that the asymptotic entropy
equals the rate of escape with respect to a distance function in terms of Green
functions. This also implies that the asymptotic entropy
equals the rate of escape with respect to the Green metric (introduced by
Blach\`ere and Brofferio \cite{blachere-brofferio}), which is given by
$d_G(x,y)=-\log F(x,y)$, where $F(x,y)$ is the probability of ever hitting $y$
when starting at $x$. The technique of our proof was motivated by
\textsc{Benjamini and Peres} \cite{benjamini-peres94}, where it
is shown that for random walks on finitely generated groups the entropy equals
the rate of escape w.r.t. the Green metric.
\textsc{Blach\`ere, Ha\"issinsky and Mathieu}
\cite{blachere-haissinsky-mathieu} generalized this result to random walks on
countable groups. We also want to mention the work of Bj\"orklund
\cite{bjorklund}, who gave an interpretation of the
Green metric in terms of Hilbert metrics.
Our result also includes an
explicit formula for the entropy and shows that the entropy can be
computed along almost every sample path. Furthermore, we get
convergence in $L_1$ of $-\frac{1}{n}\log \pi_n(X_n)$ to $h$, see \mbox{Theorem \ref{thm:entropy3}.}

The paper is organized as follows. In Section \ref{sec:preli} we
give the basic notations and definitions concerning graphs and
trees. The main results, Theorem \ref{thm:rec}, Theorem
\ref{thm:rateofescape}, and Theorems \ref{th:entropy}, \ref{thm:entropy3}, together
with  examples and discussions are presented in Section
\ref{sec:results}. All proofs are given in \mbox{Section \ref{sec:proofs}.}

\section{Preliminaries}\label{sec:preli}
Most of the time we follow the notation of
\cite{lyons:book}, where the reader can find the basic definitions
and results concerning graphs and random walks on trees. For general information on non-negative matrices we refer to \cite{seneta81} and on Banach lattices and positive operators to \cite{schaefer}.

\subsection{Notations and Definitions}\label{subsec:not}
Let  $G=(V,E)$ be a directed graph with countable vertex set $V$,
edge set $E$ and root $\orig$. For ease of presentation, we also identify a graph with its vertex set, i.e.,
$x\in G$ means $x\in V$. The \emph{adjacency matrix}
$A=\bigl(a(i,j)\bigr)_{i,j\in G}$ of $G$ is defined by $a(i,j)=1$
if there is an edge from $i$ to $j$ and $0$ otherwise. A graph is
called \emph{locally finite} if the row and column sums of $A$ are
finite. A graph is (strongly) \emph{connected} if there is a directed path
from every vertex to any other vertex, and it has \emph{bounded
geometry}  if the vertex degrees are uniformly bounded. A
\emph{tree} $\T$ is an undirected graph in which every  two
vertices are connected by exactly one shortest path. A \emph{rooted tree}
is a tree with a distinguished vertex, the root $\orig$. This
endows the tree with a natural orientation: towards or away from
the root. Furthermore, denote by $|x|$ the natural distance from a
vertex $x$ of $\T$ to $o$, i.e., the length of the shortest path (we call such
a shortest path also a \textit{geodesic})
from $\orig$ to $x$. Let $\T_{n}$ be the set of vertices at
distance $n$ from the root $\orig$. Every vertex $x\in\T_n$ has a
unique geodesic $\langle \orig=x_0,x_1,\ldots,x_n=x\rangle$ coming
from the origin $\orig$. The vertex $x_{n-1}$ is called the
\emph{ancestor} $x^-$ of $x$ and $x$ is called the \emph{direct
descendent} or \emph{successor} of $x^-$. In general, a vertex $y$
is a \emph{descendent} of $x$ if $x$ lies on the unique geodesic
from $\orig$ to $y$.  A \emph{ray} $\xi=\langle o=x_0, x_1,\ldots
\rangle$ is an infinite path from $\orig$ to infinity that doesn't
backtrack, i.e., $x_i\neq x_j$ for all $i\neq j$. We call the set
of all rays of $\T$ the \emph{(end) boundary} of $\T$, denoted by
$\partial\T$. There is a natural metric on $\partial T:$ if two
rays $\xi$ and $\eta$ have exactly $n$ edges in common  their
distance is defined as $d(\xi,\eta):=e^{-n}$. A \emph{flow}
$\theta$ is a non-negative function on the vertices of $\T$ such
that $\theta(x) = \sum_{y\in\T, y^-=x} \theta(y)$. A flow $\theta$
is called a \emph{unit flow} if $\theta(o)=1$. The \emph{lower}
and \emph{upper growth rates} of a tree $\T$ are defined as
$$
\underline{gr}(\T):=\liminf_{n\to\infty} |\T_n|^{1/n}\ \mbox{ and } \
\overline{gr}(\T):=\limsup_{n\to\infty} |\T_n|^{1/n},
$$
where $|\T_n|$ is the cardinality of the set $\T_n$. If these numbers are
equal we speak of the \emph{growth rate} $gr(\T)=\lim_{n\to\infty}
|\T_n|^{1/n}$ of $\T$. Another method to measure the growth of a
tree is the \emph{branching number} $br(\T)$. We recall two
possible definitions. The first uses the concept of flows while
the second uses the Hausdorff dimension $\dim \partial \T$ of the
boundary $\partial \T$ of the tree $\T:$
\begin{eqnarray*}\label{eq:def_br}
br(\T)  &:= & \sup \left\{ \lambda>0 \,\Bigl|\, \exists\mbox{ flow
} \theta \ \forall x\in\T:~0\leq
 \theta(x)\leq \lambda^{-|x|}\right\}\\&:=&\exp\, \dim \partial \T.
\end{eqnarray*}
Let us remark that there is the following general connection between the
lower growth rate and the branching number
%\Midline{(compare with Lyons)} \verb"It's the first Exercise it his book!":
\begin{equation}\label{eq:br_gr}
br(\T)\leq \underline{gr}(\T).
\end{equation}

\subsection{Non-negative Infinite Matrices}
Let  $M:=\bigl(m(x,y)\bigr)_{x,y\in G}$ be an infinite matrix with
non-negative entries. For $n\in\N$, let
$M^n=\bigl(m^{(n)}(x,y)\bigr)_{x,y\in G}$ be the $n$-th matrix power of
$M$ and set $M^0:=I$, the identity matrix over $\N$.
A non-negative matrix $M$ is called irreducible if  for all $x,y\in G$ there
exists some $k\in\N$ such that $m^{(k)}(x,y)>0$. We will assume
throughout the paper that there exist constants $c,C>0$
such that
\begin{equation}\label{eq:assumption}
0<c< \sum_{y\in G} m(x,y) < C <\infty \quad \textrm{ for all } x\in G.
\end{equation}
Due to the upper bound  the matrix $M$ can be interpreted as a
bounded linear operator on $\ell_p,~ p\in[1,\infty]$. For each
$f\in\ell_p$, let $Mf(x)=\sum_{y\in G} m(x,y)f(y)$.  The spectrum
of $M$ is $\sigma(M,\ell_p) := \bigl\lbrace \lambda\in\C\, \mid \,
\lambda I-T\mbox{ is not a bijection of } \ell_p\bigr\rbrace $. It
is compact and non-void. Denote by
$$
r_p(M):=\sup\bigl\lbrace |\lambda|\,\bigl|\,\lambda\in
\sigma(M,\ell_p)\bigr\rbrace =\lim_{n\to\infty} \sqrt[n]{\|M^n\|_p}
$$
the $\ell_p$-spectral radius of $M$. Observe that in general the
spectral radius is not an eigenvalue, and hence no direct
generalization of the Perron-Frobenius eigenvalue from the finite
dimensional case exists. A priori it is not clear if some $\ell_p$-spectral radius may serve as an analogon for the Perron-Frobenius eigenvalue in the finite dimensional setting. Indeed, the following characteristic turns out to be appropriate:
$$
\lambda^+(M):=\sup\bigl\lbrace \lambda>0 \,\bigl|\, \exists\, 0< f\in
\ell_\infty:~  Mf\geq \lambda f\bigr\rbrace.
$$
This number is also known as the \emph{upper Collatz-Wielandt
number}, compare with \cite{foerster:98}. The number
$\lambda^+(M)$ is well-defined, since under the general
assumption (\ref{eq:assumption}) we have
$$
c \leq \inf_{x\in G} \sum_{y\in G} m(x,y) \leq \lambda^+(M) \leq \sup_{x\in G}
\sum_{y\in G} m(x,y) \leq C.
$$
The lower bound for $\lambda^+(M)$ is obtained by investigating
$M\1$, where $\1$ is the vector with all entries equal to $1$. The
upper bound is obvious, since with $f\in\ell_\infty$ such that
$\|f\|_\infty =1$ we obtain $Mf (x) \leq \sum_{y\in G} m(x,y)$.
Furthermore, it is easy to see that the upper Collatz-Wielandt
number is less than or equal to the $\ell_\infty$-spectral radius of
$M$, that is,
$$
\lambda^+(M) \le r_\infty(M).
$$
If $M$ is homogeneous in the sense of quasi-transitiveness, we
have that $\lambda^+(M) = r_\infty(M)$, compare with Example
\ref{ex:quasi}. We also refer to \cite{vonBelow} where several other relations between the $\ell_p$-spectral radii and $\lambda^+(M)$ for symmetric matrices are discussed.

\begin{rem}\label{rem:Z}
The $\sup$ in the definition of $\lambda^+(M)$ may or may not
be attained.  Consider  the adjacency matrix $A_\Z$ of the graph
$G=\Z$. By the latter we mean the graph $G$ with $V=\Z$ and
$E=\bigl\lbrace (x,y)\in\mathbb{Z}^2 \,\mid
\,|x-y|=1\bigr\rbrace$. Clearly, $\lambda^+(A_\Z)=2$ and the
$\sup$ is attained with the vector $\1$. On the other hand,
consider $G=\N$ with its adjacency matrix $A_\N$. We still have
that $\lambda^+(A_\N)=2$. This can be seen using a recurrence
argument or observing that $r_2(A_\N)=2$ and $r_2(A_\N)\leq
r_\infty(A_\N)\leq 2$. Assume that $\lambda^+(A_\N)$ is attained
using the function $f$. Since $f(1)\geq 2 f(0)$ and $f(n+1)\geq
2 f(n) -f(n-1)$ for all $n\geq 1$ we see that $f$ is unbounded and
hence obtain a contradiction.
\end{rem}

\begin{rem}\label{rem:gen:perron}
Another characteristic that might serve as a generalization of the Perron-Frobenius eigenvalue to the infinite setting is $\rho(M):=\limsup_{n\to\infty}
\left(m^{(n)}(x,y)\right)^{1/n}$. The latter is independent of the specific choice of
$x,y$ and equals the $\ell_2$-spectral radius if $M$ is symmetric, e.g. compare
with \cite[Chapter II, Section 10]{woess}. It can be seen analogously to von
Below \cite[Corollary 4.6]{vonBelow} that
$\rho(M) \leq \lambda^+(M)$. 
Furthermore, $\rho(M)$ can be given
in terms of convergence parameters of
Green functions, rate functions of large deviation principles, and
super-harmonic functions.
Moreover, $\rho(M)$ depends on local
properties, since
\begin{equation}\label{eq:spec_approx}
 \rho(M)=\sup_{F\subset G, |F|<\infty} \rho (M_F),
\end{equation}
where $M_F=\bigl(m_F(x,y)\bigr)_{x,y\in F}$ is the matrix induced
by $F$, i.e., $m_F(x,y)=m(x,y)$ for $x,y\in F$ and $0$ otherwise.
This fact indicates that $\rho(M)$ is not the  \emph{good}
characteristic for recurrence and transience, since these
properties describe a global behaviour.
\end{rem}

\begin{rem}
Let $A$ be the adjacency matrix of a graph $G$, which is symmetric, that is, there is an
edge from $i\in G$ to $j\in G$ if and only if there is an edge from $j$ to $i$. While in the
finite case we always have $r_2(A) =r_\infty(A)$, in general we only
have $r_2(A)\leq r_\infty(A)$ for infinite $A$.
Let $G=\T_d$ be the regular tree with degree $d\geq 3$. Clearly, we
have $r_\infty=d$ and it is very well-known that $r_2(A)= 2
\sqrt{d-1}$. It is worth noting that if
$r_2(A)<r_\infty(A)$ we can not approximate $r_\infty(A)$ using Perron-Frobenius eigenvalues of finite subgraphs,
compare with Equation (\ref{eq:spec_approx}).
\end{rem}

\section{Results}\label{sec:results}

\subsection{Trees as Directed Covers of Infinite Graphs of Bounded Geometry}\label{subsec:trees}
Suppose we are given a connected and directed graph $G=(V,E)$ of
bounded geometry without multiple edges. In
some cases, e.g. Remark \ref{rem:loc:infinite}, we can drop the
assumption of bounded geometry. Let $i_0\in V$ be a distinguished vertex of $G$.
The \emph{directed cover} $\T$ of
$G$ (rooted in $i_0$) is defined recursively as a rooted tree
$\mathcal{T}$, whose vertices are labelled by the vertex set $V$. The root $\orig$
of $\T$ has label $i_0$; recursively, if $x\in \mathcal{T}$ is
labelled with $i$, then $x$ has one direct descendent with label
$j$ if and only if there is an edge from $i$ to $j$ in $G$. And
vice versa we define the \emph{label function}  $\tau:\T\to V$ to
be the function that associates to each vertex in $\T$ its label
in $V$. In order to distinguish between the two graphs $G$ and $V$
we use in general the variables $i,j$ for vertices in $G$ (and
labels in $T$) and $x,y$ for the vertices of $\T$.  We refer to
\cite{lyons90} and \cite{nagnibeda:02} for more details and
references for finite graphs $G$. In the finite case these trees
are known as \emph{periodic trees} (\cite{lyons90}, \cite{takacs})
or \emph{trees with finitely many cone types}
(\cite{nagnibeda:02}).

We remark that we may assume w.l.o.g. that the graph has no
multiple edges. Indeed, suppose that $G$ has multiple edges. Then one can
replace $G$ by a new graph $G'$, which arises from $G$ by introducing new
vertices and edges and deleting those multiple edges. That is, if
there are $m>2$ edges from $i$ to $j$ in $G$, then we add to $G$ the new
vertices $j_1,\dots,j_{m-1}$ and add an edge from $i$ to each of these new
vertices. Each of these new vertices is again connected with every vertex which
has an edge coming from $j$. Repeating this operation gives a new graph $G'$, which has
now no multiple edges. 

Note that there is a one-to-one correspondence between all finite
paths $\langle i_0, i_1, i_2,\ldots i_n\rangle$ in $G$ starting at
$i_0$ and all vertices in $\T$. Each vertex $x_n\in\T$ at height
$n$ corresponds to a unique geodesic $\langle \orig, x_1,
x_2,\ldots, x_n\rangle$ coming from the origin. This geodesic is
uniquely determined by its labels, i.e., the path  $\langle i_0,
\tau(x_1), \tau(x_2),\ldots, \tau(x_n)\rangle$ in $G$.
\par

The \emph{cone} $\T^x:=\{y\in \T\,\mid \, y~\mbox{ is a descendent
of } x\}$ is the subtree rooted at $x$ and spanned by all vertices
$y$ such that $x$ lies on the geodesic from $o$ to $y$. We say that $\T^x$ has
\emph{cone type} $\tau(x)$. Observe that if $x,y\in\T$ with $\tau(x)=\tau(y)$
then the
trees $\T^x$ and $\T^y$ are
isomorphic as rooted trees.
\par
%Since a priori every tree has infinitely many cone types we will
%assume that they are irreducible, i.e., every cone type contains
%any other cone type as a subtree. Clearly, this is guaranteed for
%directed covers of connected graphs. It is easily seen that every
%tree with irreducible cone types may arise as the directed cover
%of an appropriate directed connected graph.

%\begin{rem}
%Note that every tree with irreducible cone types may arise as a
%directed cover of a (not necessarily unique) graph. The
%construction of such a graph is quite intuitive. We start with the
%\emph{empty} graph $G$, i.e., $V=E=\emptyset$. Now, we add a
%vertex $i_0$ to $G$ that corresponds to the cone type of
%$\orig\in\T$.  For each descendant of the root whose cone is not
%isomorphic to the cone of $\orig$ we add new vertices  to $G$.
%Furthermore, we add corresponding edges from $i_0$ to these new
%ones. If a vertex has the same cone type as the root $\orig$ we
%just add a loop to $0$ and will disregard all vertices of this
%cone in the following construction. Following this idea the graph
%$G$ is defined inductively by adding vertices to $G$ if the cone
%type has not yet been considered and disregarding the vertices of
%a \emph{known} cone type. Additionally, we add edges to these new
%vertices or to the already existing vertices. If a vertex in $\T$
%has two direct descendants of the same cone type, then one can
%allow $G$ to have multiple edges or one adds a new vertex to $G$
%which represents the same cone type but different label.
%\end{rem}

For general trees we have $br(\T)\leq \underline{gr}(\T)$. While
for directed covers of finite graphs $br(\T)=gr({\T})$ the growth
rate does not exist in general, compare with  Example
\ref{ex:no_gr}. Nevertheless we have the following result for
directed covers of infinite graphs.

\begin{thm}\label{thm:growth}
Let $A$ be the adjacency matrix of $G$ and $\T$ be the directed
cover of $G$. Then:
\begin{enumerate}
  \item $\lambda^+(A)= br(\T) = \underline{gr}(\T)$.
  \item $\overline{gr}(\T)
\leq r_\infty(A)$.
\end{enumerate}
\end{thm}

\begin{ex}\label{ex:no_gr}
We give an example where the growth rate does not exist.
The desired graph $G$ will be constructed inductively. To this end,
let $(k_n)_{n\in \N}$ be the sequence defined by $k_1:=2$ and
$k_{n+1}:=3 \sum_{j=1}^n k_j$.  For $n$ odd, let $G_n$ be the
\emph{circle} with $k_n$ vertices $\{1,2,\ldots,k_n\}$ and
directed edges $(i,i+1)$ for $1\leq i< k_n$ and $(k_n,1)$. We call
$1$ the starting point and $k_n$ the end point of the circle. For
$n$ even,  $G_n$ is constructed as follows. Let $T_2^{k_n-1}$ be a
rooted binary tree of height $k_n-1$ with directed edges all leading
away from the root. The leaves at level $k_n-1$
are enumerated by $\{1,2,\ldots, 2^{k_n-1}\}$. Now we add direct
paths of length $k_{n +1}$ from leaf $i$ to leaf $i+1$ for
$1\leq i < 2^{k_n-1}$ and one directed path from leaf $2^{k_n-1}$
to $\orig_n$ of length $k_{n+1}$. Here, $\orig_n$ is an additional
vertex with a directed edge from $\orig_n$ to the root of
$T_2^{k_n-1}$.  Observe that all graphs $G_n$ are connected. The
graph $G$ is now defined inductively. We start with $G_1$. We add
a copy of $G_2$ in \emph{gluing} (vertices are identified) the end
point of $G_1$ with  $\orig_2$ of $G_2$. At each leave of $G_2$ we
glue a copy of $G_3$ and continue inductively. Observe that each
vertex has outdegree at most $2$ and the vertices (except the
starting and end point) in copies of the circle $G_n$ ($n$ odd)
have only outdegree $1$ as well as the vertices on the directed
paths leading away from the leaves of $G_n$ ($n$ even). Let
$\beta_n=\sum_{i=1}^n k_i$. Therefore, for $n$ even we find that
$$
|T_{\beta_n}|\geq 1^{\frac{\beta_{n}}4} 2^{\frac{3\beta_{n}}4}.
$$
Due to the construction of $G$ every path emanating from the
origin  of length $\beta_{n}$ for $n$ odd has at most
$\beta_{n}/2$ vertices of degree $2$. This yields for $n$ odd
$$
|T_{\beta_{n}}|\leq 2^{\frac{\beta_n}2}.
$$
Eventually,
\begin{equation}
 \underline{gr}(\T)\leq 2^{\frac12}<2^{\frac34}\leq \overline{gr}(\T).
\end{equation}
\end{ex}

\subsection{Recurrence and Transience of Random Walks on Directed Covers}\label{subsec:type}
We consider the  model of nearest neighbour random walks on
$\mathcal{T}$ according to  \cite{nagnibeda:02}. Suppose we are
given transition probabilities $p_G(i,j)$ on $G$, where $(i,j)\in
E$. We hereby assume that
$p_G(i,j)>0$ if and only if there is an edge from $i$ to $j$ in
$G$. Furthermore, suppose we are given \emph{backward}
probabilities $p(-i)\in (0,1)$ for each $i\in G$.  For ease of presentation and
technical reasons, we add to $\T$ a loop at the
origin $\orig$.  Then the random walk on the tree $\T$ is
defined through the following transition probabilities $p(x,y)$,
where $x,y\in\T$:
\begin{equation}
p(\orig,y):=
\begin{cases}
\bigl(1-p(-i_0)\bigr) p_G\bigl(i_0,\tau(y)\bigr), & \mbox{ if }y\neq \orig,  \\
p(-i_0), &  \mbox{ if }y=\orig,
\end{cases}
\end{equation}
and for $x\neq \orig$ with $\tau(x)=i$,
\begin{equation}
p(x,y):=
\begin{cases}
\bigl(1-p(-i)\bigr) p_G\bigl(i,\tau(y)\bigr), & \mbox{ if }x=y^-,  \\
p(-i), &  \mbox{ if }y=x^-.
\end{cases}
\end{equation}
We will also write $p(i,j):=\bigl(1-p(-i)\bigr)p_G(i,j)$ and
$p(-i):=p(x,x^-)$, where $\tau(x)=i$, for the transition probabilities in
$\mathcal{T}$. The random walk on $\T$ starting at $\orig$ is
denoted by the sequence of random variables $(X_n)_{n\in\N_0}$.
We are going to characterize recurrence and transience in terms of
$\lambda^+(M)$, where
\begin{equation}\label{eq:def_M}
M = \bigl( m(i,j) \bigr)_{i,j\in G} \ \textrm{ with }
m(i,j) := \frac{ p(i,j)}{p(-i)}.
\end{equation}
Observe that our notation differs from the one in
\cite{nagnibeda:02} where the finite analogue of the matrix $M$ is
denoted by $A$. In \cite{nagnibeda:02} it is proved that for finite $G$
 the random walk on $\T$ is positive recurrent if
$\lambda(M)<1$, null-recurrent if $\lambda(M)=1$, and transient if
$\lambda(M)>1; $  here $\lambda(M)$ is the Perron-Frobenius
eigenvalue of $M$. This classification  was also obtained in
\cite{gairat} for the more general model of random strings. The
matrix used in \cite{gairat} is different from the one in
\cite{nagnibeda:02}  but leads to the same results.

As an additional assumption we demand that there is
some $\eps\in (0,1)$ such that
\begin{equation}\label{equ:p-i}
\eps< p(-i)< 1-\eps \quad \textrm{ for all } i\in G.
\end{equation}
This assumption is necessary to exclude degenerate examples and it
assures  assumption (\ref{eq:assumption}) on the infinite matrix $M$.

Due to the tree structure  the random walk is reversible with
reversible measure  $m$ defined recursively by
\begin{equation}\label{eq:rev_m}
m(\orig):=1 \ \mbox{ and } \ m(x):=m(x^-)\frac{p(x^-,x)}{p(x,x^-)}\
\mbox{ if } x\neq \orig.
\end{equation}

Hence, the random walk can be considered  as an electric network, see
\cite{dyole:84} or \cite{lyons:book} for general information on electric networks. Recall that transience of a
subnetwork implies transience of the larger network,
compare with Chapter 2 in \cite{lyons:book}. Let
$$
\rho(M):=\limsup_{n\to\infty} \left(m^{(n)}(x,y)\right)^{1/n}
$$
be the
spectral radius of $M$.  If $\rho(M)>1$ then there exists a
finite set $F$ such that $\lambda(M_F)>1$ and hence we conclude
from the results of \cite{nagnibeda:02} that the walk restricted
to the cover of $F$ is transient. Therefore, we have that
$\rho(M)>1$ implies transience of the random walk. The reverse is
not true in general, compare with Theorem \ref{thm:rec} and Remark \ref{rem:gen:perron}.

Since the methods used in \cite{gairat} and \cite{nagnibeda:02}
cannot  directly be generalized to the infinite setting we present
a new approach for proving  recurrence and transience. This
provides an alternative proof in the finite case.  The idea is to
couple an infinite-type generalized Galton-Watson process (in
continuous time) to the random walk (in discrete time) such that
the Galton-Watson process dies out  if and only if the random walk
on $\T$ visits the edge $(\orig, \orig)$. To this end observe that
$M$ is a non-negative (infinite) matrix. Thus, $M$ can be
interpreted as the first moment matrix of an infinite-type
Galton-Watson process. That means that $m(i,j)$ is the mean number
of particles of type $j$ that one particle of type $i$ produces in
its lifetime. Let us define a process $Z_t$ with types indexed by
$G$ with first moment matrix $M$.
The process is described through the number $Z_t(i)$ of particles of type $i$ at time $t$ and evolves
according the following rules: for each $i\in G$
\begin{eqnarray*}
  Z_t(i) &\rightarrow & Z_t(i)-1 \mbox{ at rate }  Z_t(i) p(-i), \\
  Z_t(i) &\rightarrow & Z_t(i) +1 \mbox{ at rate }  \sum_{j}  Z_t(j) p(j,i).
\end{eqnarray*}

In words, each particle of type $i$ dies at rate $p(-i)$
and gives birth to new particles of type $j$ at
rate $p(i,j).$
Let $Z_t=\bigl(Z_t(i)\bigr)_{i\in G}$ describe the whole population of
the process. We consider the probability of (global) extinction
$q:=\P[\exists t: Z_t =\mathbf{0}]$  and make the following crucial
observation:

\begin{thm}\label{thm:br_mc}
The extinction probability $q$ of the process $Z_t$ equals the
probability that the random walk on $\T$ visits the edge
$(\orig,\orig)$.
\end{thm}

Now, we use a result on infinite-type Galton-Watson processes of
\cite{bertacchi08} in order to obtain the classification result:

\begin{thm}\label{thm:rec}
The random walk $(X_n)_{n\in\N_0}$ on $\T$ is recurrent if $\lambda^+(M)<1$ and it is
transient if \mbox{$\lambda^+(M)>1$.} In the critical case,
$\lambda^+(M)=1$, it may be transient or recurrent.
\end{thm}

\begin{ex}\label{ex:lambda=1}
We consider Example 3 of
\cite{bertacchi08} that was given in terms of infinite-type
Galton-Watson processes. Let $G:=\N_0$ with edges of the form $(i,i+1)$ and
$(i,i-1)$ for $i\geq 1$, including the edge $(0,1)$.
Let $p(-0):=1/3,~ p(0,1):=2/3$ and
$$
p(-i):=\left(1+\left(1+\frac1i\right)^2 +
\left(\frac13\right)^i\right)^{-1},
$$
$$
p(i,i+1):=p(-i) \left(1+\frac1i\right)^2 \mbox{ and }\
p(i,i-1):=p(-i) \left( \frac13\right)^i \mbox{ for } i\geq 1
$$
define the random walk on $\T$. Hence the values of  $M$ are
$m(0,1)=2$ and
$$
m(i,i+1)=\left(1+\frac1i\right)^2\mbox{ and }
m(i,i-1)=\left(\frac13\right)^i\mbox{ for } i\geq 1
$$
and $0$
otherwise. Observe now that the function $g:\N_0\to \R$ defined by $g(0):=1$ and $g(i):=i/(i+1)$ for $i\geq
1$ is a solution for the inequality $M g\geq g$. Furthermore, one
can easily show by induction that $M g\geq \lambda g$ with $\lambda>1$ implies that either
$g\equiv 0$ or $g(i)\to\infty$ for $i\to\infty$. Eventually,
$\lambda^+(M)=1$. In order to show that the random walk on $\T$ is
transient we use a coupling argument. We compare the original
process with the Markov chain on the positive integers with transition
probabilities 
$\tilde p(0,1):=1$, and $\tilde p(i,i+1):=p(i,i+1),~ \tilde
p(i,i):=p(i,i-1)$, and $\tilde p(i,i-1):=p(-i)$ for $i\geq 1$.
Hence, using a coupling argument, the random walk on the directed
cover is transient if the Markov chain $(\tilde X_n)_{n\in\N_0}$ on $\N_0$ with
transition probabilities $\tilde p(\cdot,\cdot)$ is transient. To
see the latter, observe that the mean drift is
$$
\mu_1(i):=\E \bigl[ \tilde X_{n+1}-\tilde X_n \bigl| \tilde X_n=i\bigr]=p(-i)\left(\frac2i
  +\frac{1}{i^2}\right)
$$
and
$$
\mu_2(i):=\E \bigl[ (\tilde X_{n+1}-\tilde X_n)^2 \bigl| \tilde X_n=i\bigr]\leq 1.
$$
Hence,  for $i$ sufficiently large we obtain
$$
\mu_1(i)\geq  \frac2{3i}\geq \frac{2\,\mu_2(i)}{3i}
$$
and conclude from Theorem 3.6.1 in \cite{fayolle} that the random walk is
transient.
\end{ex}

\begin{ex}
For any given graph $G$, a straightforward example for recurrence in the
critical case $\lambda^+(M)=1$ is given by $p(-i):=1/2$ for every $i\in G$. The random walk on $\T$ can then be
naturally projected on $\N_0$, that is, $X_n$ will be projected on
$|X_n|$. In this case, the Markov chain $(|X_n|)_{n\in\N_0}$ is
null recurrent, and thus, $(X_n)_{n\in\N_0}$ is also
null-recurrent.
\end{ex}

\subsection{Examples}\label{subseq:ex}

\subsubsection{Homesick Random Walks}
A special class of random walks on trees are homesick random
walks, compare with \cite{lyons:book} and \cite{takacs}. In this
model the edge leading back towards the root is $\lambda$ times as
likely to be taken as each other edge, that is, $p(x,x^-)=\lambda/
(\lambda+deg(x)-1)$ and $p(x^-,x)=1/ (\lambda+deg(x^-)-1)$, where
$deg(x)$ is the number of edges adjacent to $x$. We denote the
dependence of the random walk on the parameter $\lambda$ by
$RW_\lambda$. In \cite{lyons90} it is shown that the $RW_\lambda$
is recurrent if $\lambda>br (\T)$ and transient if $\lambda< br
(\T)$. This result holds for any tree $\T$, not necessarily
directed cover.

Given an underlying simple random walk on a graph $G$  the
homesick random walk is the random walk on the directed cover with
$p(-i)=\lambda/ (\lambda + outdeg(i))$, where $outdeg(i)=\sum_j
a(i,j)$ is the number of outgoing edges from $i\in G.$ It is easy
to see that $\lambda^+(M)=\lambda^+(A)/\lambda,$ where $A$ is the
adjacency matrix of $G$ and $M$ is defined as in Equation
(\ref{eq:def_M}). As a consequence of Theorem \ref{thm:rec} we
obtain the following:

\begin{cor}\label{cor:homesick}
Let $A$ be the adjacency matrix  and $\T$ the directed cover of
$G$. Then the homesick random walk $RW_\lambda$ on $\T$ is
recurrent if $\lambda>\lambda^+(A)$ and transient if
$\lambda<\lambda^+(A)$.
\end{cor}
While $\lambda=\lambda^+(A)$ implies recurrence if $G$ is finite
(compare with Theorem 3.5 in \cite{lyons:book}), the behaviour for
infinite graphs $G$ is not known in general.

\subsubsection{Directed Covers of Quasi-Transitive Graphs}\label{ex:quasi}

Suppose we are given a locally finite graph $G$ with bounded
geometry. Denote by $\mathrm{AUT}(G)$ the group of all
automorphisms $\gamma$ of the vertex set of $G$ which leave the
adjacency relation invariant, that is, there is an edge from
$\gamma(i)$ to $\gamma(j)$ if and only if there is an edge from
$i\in G$ to $j\in G$. Assume now that $G$ is
\emph{quasi-transitive}, that is, $\mathrm{AUT}(G)$ acts with
finitely many orbits on the vertex set of $G$. We write
$\mathrm{Orb}:=\{o_1,\dots,o_r\}$ for the set of orbits. We
construct a new finite graph $G'$ with vertex set $\mathrm{Orb}$
in the following natural way: there are $d(i,j)\in\N_0$ edges from
$o_i$ to $o_j$ if and only if there are $d(i,j)$ edges in $G$
from some $k\in o_i$ to some $l\in o_j$. Now we can apply the
well-known results from the finite setting, since $G'$ and $G$
create the same (unlabelled) cover $\T$. Thus, $br(\T)$, which
becomes the largest eigenvalue of the adjacency matrix of $G'$,
equals $\lambda^+(A)$. Furthermore, it is easy to see that for quasi-transitive graphs we have $\lambda^+(A)=r_\infty(A)$.

Now, suppose  we are given a  random walk on $G$ governed by the
transition matrix $P_G$. The set $\mathrm{AUT}(G,P_G)$ is the
group of all automorphisms $\gamma$ of the vertex set of $G$,
which leave $P_G$ invariant, that is,
$p_G\bigl(\gamma(i),\gamma(j)\bigr)=p_G(i,j)$. Then $(G,P_G)$ is
called \emph{quasi-transitive} if $\mathrm{AUT}(G,P_G)$ acts with
finitely many orbits on the vertex set of $G$. Assume now that
$(G,P_G)$ is quasi-transitive and that the backward probabilities
$p(-i)$ are constant on the orbits of $\mathrm{AUT}(G,P_G)$.
Denote by $\mathrm{Orb}:=\{o_1,\dots,o_r\}$ the orbits of
$\mathrm{AUT}(G,P_G)$. We define a random walk on $\mathrm{Orb}$
by setting $\tilde p(o_i,o_j):=\sum_{l\in o_j} p_G(k,l)$, where
$k\in o_i$ is arbitrary. Then $\lambda^+(M)$ is the
Perron-Frobenius of the matrix $\bigl(\tilde
p(o_i,o_j)/p(-o_i)\bigr)_{1\leq i,j\leq r}$. This follows from the
results of \cite{nagnibeda:02}. In particular, the random walk is
null-recurrent if $\lambda^+(M)=1$.

\subsubsection{Directed Covers of Percolation Clusters}\label{ex:per_clu}
We consider supercritical Bernoulli$(p)$ percolation on $\Z^d$,
i.e., for fixed $p\in [0,1]$, each edge is kept with probability
$p$ and removed otherwise, independently of the other edges. It is
well-known that there exists a critical value $p_c$ such that for
$p<p_c$ there is almost surely no infinite connected component and for $p>p_c$
there is almost surely exactly one infinite connected component. In the latter
case we denote by $C_{\omega}$ the infinite connected component for
the realization or \mbox{environment $\omega$.} We refer to \S 6 in
\cite{lyons:book} for more information on percolation models. Let
$A_\omega$ be the adjacency matrix  and $\T_{\omega}$ be the
directed cover of the infinite cluster $C_{\omega}$ with respect to some
$i_0\in C_{\omega}$. We get the following classification:
\begin{thm}\label{thm:rw_perco}
The homesick random walk $RW_\lambda$ on almost every directed cover
$\T_{\omega}$ is recurrent if $\lambda\ge 1/2d$ and  transient if
$\lambda< 1/2d$.
\end{thm}

\subsubsection{Random Walk on Directed Covers in Random
Environment}

We consider the nearest neighbour random walk in random
environment on $V=\Z$ and edge set $E=\bigl\lbrace (x,y) \,\mid \,
|x-y|=1\bigr\rbrace$. We choose i.i.d. random variables
$\omega_z^+$ ($z\in\Z$) with values in $(0,1)$ and  call $\eta$
the distribution of the environment with  one-dimensional marginal
\mbox{distribution $\theta$.} For a given realization  $\omega$ of this
random environment, we consider the Markov chain on $\Z$ with
transition kernel $P_\omega$ defined as

$$
p_{\Z,\omega}(z,z+1):=\omega_z^+ \ \mbox{ and }\
p_{\Z,\omega}(z,z-1):=\omega_z^-:=1-\omega_z^+\quad \textrm{ for all } z\in\Z.
$$
We refer to \cite{zeitouni2001} for  details on this model. In
addition, we introduce an environment that defines the backwards
probabilities. So let $\nu_z$ ($z\in\Z$) be i.i.d. random
variables with values in \mbox{$(\eps,1-\eps)$} for some $\eps\in(0,1)$.
We call $\tilde\eta$ the distribution of this environment and
denote by $\tilde\theta$ its one-dimensional marginal distribution.
Every given realization $\nu=(\nu_z)_{z\in\Z}$ determines the
backwards probabilities by $p_\nu(-z):=\nu_z$. Let $\Theta$ be the
corresponding product measure with one-dimensional marginal
$\theta \times \tilde \theta$. Every given realization
$(\omega,\nu)$ defines a random walk on a directed cover in random
environment (RWDCRE) with corresponding matrix $M_{\omega,\nu}$.
The classification in recurrence and transience will be stated in
terms of the top Lyapunov exponent of sequences of random
matrices. For $k\in\{1,2,3,\ldots\}$, we write
\begin{equation}
A_k:=\left(
             \begin{array}{cc}
               \nu_k \omega_k^+ & -\frac{\omega_k^-}{\omega_k^+} \\
               1 & 0 \\
             \end{array}
           \right) \mbox{ and }\
\tilde A_k:=\left(
             \begin{array}{cc}
               \nu_k \omega_k^- & -\frac{\omega_k^+}{\omega_k^-} \\
               1 & 0 \\
             \end{array}
           \right).
\end{equation}
Denote by $\gamma_1$  the top Lyapunov exponent associated with the sequence
$(A_k)_{k\in\N}$, i.e.,
$$ \gamma_1 =\lim_{n\to \infty} \frac1n  \E \bigl[ \ln \|A_n \cdots A_1\| \bigr],$$
where $\|\cdot\|$ is any matrix norm. Analogously, let
$\tilde\gamma_1$ be the top Lyapunov exponent of the sequence
$(\tilde A_k)_{k\in\N}$. For sake of better readability, we
set $\mu_i^-=\omega_i^-/ \nu_i$ and $\mu_i^+=\omega_i^+/ \nu_i$.
 Combining Theorem \ref{thm:br_mc} with
Theorem 2.6 and Theorem 2.9 in \cite{gantert:08} yields the
following classification. (Observe that $\mu_i^{0}$ of
\cite{gantert:08} equals $0$ in our special case.)

\begin{thm}\label{thm:re}
We have the following classification:
\begin{enumerate}
  \item If there exists no $\lambda>0$ such that  $\mu_0^- \lambda^{-1}+
    \mu_0^+ \lambda \le 1$ $\Theta$-almost surely, then the RWDCRE is transient
    $\Theta$-almost surely.
  \item If there exists  $\lambda>1$ such that  $\mu_0^- \lambda^{-1}+ \mu_0^+
    \lambda \le 1$ $\Theta$-almost surely, then the RWDCRE is transient
    $\Theta$-almost surely if and only if 
$$
\gamma_1 < \E \ln \left(\frac{\mu_0^-}{\mu_0^+} \right).
$$
  \item If there exists  $\lambda<1$ such that  $\mu_0^- \lambda^{-1}+ \mu_0^+
    \lambda \le 1$ $\Theta$-almost surely, then the RWDCRE is transient
    $\Theta$-almost surely if and only if 
$$
\tilde \gamma_1 < \E \ln \left(\frac{\mu_0^+}{\mu_0^-} \right).
$$
\end{enumerate}
\end{thm}
\begin{rem} 
Observe that Theorem \ref{thm:re} gives a complete classification since 
$\mu_0^-+\mu_0^+=1/\nu_0 > 1$.
\end{rem}

\subsection{Ergodicity}
In order to prove (non-)ergodicity we investigate whether the
reversible measure in Equation (\ref{eq:rev_m}) is (in)finite.
This method leads to the following generalization of the finite
case.

\begin{prop}\label{prop:non_ergod}
The random walk on $\T$ is non-ergodic if $\lambda^+(M)>1$ and it
is ergodic if $r_\infty (M)<1$.  If $\lambda^+(M)=1$ and the
supremum in the definition of $\lambda^+(M)$ is attained the
process is non-ergodic as well.
\end{prop}

Since in the infinite case the random walk is \emph{more
transient} in the critical case we conjecture the random walk to
be non-ergodic for $\lambda^+(M)\geq 1$.

\subsection{Rate of Escape}\label{sec:rateofescape}

In this section we generalize the results of \cite{nagnibeda:02}
to covers of infinite graphs and more general length functions.
In order to state the main result we need the following
definitions and notations. Suppose we are given a bounded function $w:G \times G\to \R$ representing a weight for each edge in $G$. We define
recursively a length function $l$ on $\mathcal{T}$ by $l(o):=0$
and $l(x):=l(x^-)+ w\bigl(\tau(x^-),\tau(x)\bigr)$ otherwise. If
$w(\cdot,\cdot)=1$, then $l(x)$ is just the natural graph distance
on $\T$ (that is, the number of edges that connect $o$ with $x$)
denoted by $|x|$. Observe that  $w$ may  take  negative values; in
this case one can think of $w$ describing height differences between 
neighbour vertices. If there is some constant $\ell\in\R$ such
that
$$
\ell =\lim_{n\to\infty} \frac{l(X_n)}{n} \quad \textrm{ almost surely,}
$$
then $\ell$ is called the \emph{rate of escape} or \emph{drift} of
$(X_n)_{n\in\N_0}$ w.r.t. the length function $l$. We suppose for
the rest of this section that $(X_n)_{n\in\N_0}$ is
\emph{transient}. Let $i\in G$, $x,y\in\T$ with $\tau(x)=i$ and
$z\in\mathbb{C}$. We introduce the following generating functions:
\begin{eqnarray*}
F(-i|z) & := &  \sum_{n\geq 0} \P[X_n=x^-,\forall m<n: X_m\neq x^- \mid
X_0=x]\,z^n,\\
G(x,y|z) & := & \sum_{n\geq 0} \P[X_n=y \mid X_0=x]\,z^n,\\
\overline{G}_i(z) & := & \sum_{n\geq 0} \P[X_n=x,\forall m<n: X_m\neq x^- \mid
X_0=x]\,z^n.
\end{eqnarray*}
We also write  $F(-i):=F(-i|1)$. Note that the definitions of $F(-i|z)$ and
$\overline{G}_i(z)$ are independent of the specific choice of $x\in\T$ with
$\tau(x)=i$. Moreover, we have the following equations:
\begin{eqnarray}
F(-i|z) & = & p(-i)\,z + \sum_{j\in G} p(i,j)\,z\,F(-j|z)\,F(-i|z),\label{f-equations}\\
\overline{G}_i(z) & = & \frac{1}{1-\sum_{j\in G} p(i,j)\,z\,F(-j|z)}.\nonumber
\end{eqnarray}
Furthermore, we have $F(-i|z)=\overline{G}_i(z)\,p(-i)\,z$. Recall
that the spectral radius of $(X_n)_{n\in\N_0}$ is the inverse of
the radius of convergence  $R$ of $G(o,o|z)$.

Define for $k\in\N_0$ the \textit{exit times}
\begin{eqnarray}\label{def:exittime}
\e{k} & := & \min\bigl\lbrace m\in\N_0 \,\bigl|\,  \forall m'\geq
m: |X_{m'}|\geq k\bigr\rbrace
\end{eqnarray}
and write $\W{k} := X_{\e{k}}$. Observe that
$\bigl(\tau(\W{k})\bigr)_{k\in\N_0}$ is a Markov chain with
transition matrix $Q=\bigl(q(i,j)\bigr)_{i,j\in G}$ defined by
\begin{equation}\label{def:q}
q(i,j) := \frac{1-F(-j)}{1-F(-i)}   p(i,j) \overline{G}_i(1).
\end{equation}
This can be easily verified analogously to \cite{nagnibeda:02} or
\cite{gilch:07}.
Now, we can state the result about the rate of escape.
\begin{thm}\label{thm:rateofescape}
Suppose that $Q$ is positive recurrent with invariant probability measure
$\nu$. Let
$$
\Lambda := \sum_{i\in G} \nu(i)\,\frac{F'(-i|1)}{F(-i)}.
$$
Then the following hold::
\begin{enumerate}
\item If $\Lambda<\infty$  we have for each bounded weight
function $w$
\begin{equation}\label{eq:rate_of_escape}
\lim_{n\to\infty} \frac{l(X_n)}{n} = \biggl( \sum_{i,j\in G}
w(i,j) \nu(i)q(i,j)\biggr) \cdot \biggl(\sum_{i\in G}
\nu(i)\,\frac{F'(-i|1)}{F(-i)} \biggr)^{-1} \quad \textrm{ almost surely.}
\end{equation}
In this case the rate of escape w.r.t. the natural distance
exists and is positive, i.e.,
$$
\lim_{n\to\infty} \frac{|X_n|}{n}>0\quad \textrm{ almost surely.}
$$
\item
If the spectral radius of $(X_n)_{n\in\N_0}$ is strictly smaller
than $1$, then $\Lambda<\infty$.
\item
If $\Lambda =\infty$ then $\liminf_{n\to\infty} \frac{l(X_n)}{n}=0$.
\end{enumerate}
\end{thm}

\begin{rem}
The formula for the rate of escape (\ref{eq:rate_of_escape}) is of
rather formal nature, since it is given in terms of the generating
functions $F(-i|z)$. These generating functions  are solution of
the infinite system of algebraic equations (\ref{f-equations}).
\end{rem}

\begin{rem}\label{rem:loc:infinite}
Theorem \ref{thm:rateofescape}.1 and \ref{thm:rateofescape}.3 hold
for locally infinite graphs, too.
\end{rem}

The assumption of positive recurrence of $Q$ in Theorem
\ref{thm:rateofescape} is essential. In the following we give an
example, Example \ref{ex:speed2}, where $Q$ governs a transient
random walk such that the rate of escape is random. An example,
Example \ref{ex:speed1}, where $Q$ is positive recurrent is given
in \mbox{Section \ref{sec:entropy}.}

\begin{ex}\label{ex:speed2}
Consider $G=\mathbb{Z}$ with its usual neighbourhood relation and
transition probabilities given by
$$
p_G(i,i+1)=
\begin{cases}
p, & \textrm{if } i\geq 1\\
1-q, & \textrm{if } i\leq -1
\end{cases}
\textrm{, }\quad
p_G(i,i-1)=
\begin{cases}
1-p, & \textrm{if } i\geq 1\\
q, & \textrm{if } i\leq -1
\end{cases},
$$
and $p_G(0,1)=p_G(0,-1)=1/2,$
where $p,q\in(1/2;1)$, $p\neq q$.  We set $i_0:=0$. Choose now
$c_1,c_2$ with
$$
0<c_1<c_2<\min\biggl\lbrace 1-\frac{1}{2p},1-\frac{1}{2q}\biggr\rbrace.
$$
Consider the
directed cover $\mathcal{T}$ of $G$, where $p(-i):=c_1$ if $i\geq 0$, and
$p(-i):=c_2$ if $i<0$. By definition of $c_1$ and $c_2$, the random walk on $\mathcal{T}$ visits only
finitely many vertices of each cone type, since
\begin{eqnarray*}
p(i,i+1)=1-p(i,i-1)-p(-i)& = &  1-\bigl((1-c_1)(1-p)+c_1\bigr) > \frac{1}{2}, \ \textrm{ if } i\geq 0, \\
p(i,i-1)=1-p(i,i+1)-p(-i)& = & 1-\bigl((1-c_2)(1-q)+c_2\bigr) > \frac{1}{2}, \ \textrm{ if } i\leq -1.
\end{eqnarray*}
Thus, $\tau(X_k)$ tends either to $+\infty$ or $-\infty$. Considering
the speed w.r.t. the natural graph metric in the tree, in the first case the random walk has speed
$$
\E\bigl[|X_{n+1}|-|X_n|\,\bigl|\, \tau(X_n)=i>0\bigr] =(1-c_1)\,p + (1-c_1)(1-p) - c_1 = 1-2c_1,
$$
while in the case $\tau(X_k)\to-\infty$ the rate of escape is
different, namely
$$
\E\bigl[|X_{n+1}|-|X_n|\,\bigl|\, \tau(X_n)=i<0\bigr]=(1-c_2)\,q + (1-c_2)(1-q) - c_2 = 1-2c_2.
$$
\end{ex}

\subsection{Asymptotic Entropy and Hausdorff Dimension}\label{sec:entropy}

A characteristic of transient random walks that is connected to
the rate of escape is the asymptotic entropy of the process. Let
$\pi_n$ be the distribution of $X_n$, that is, for $x\in
\mathcal{T}$ the number $\pi_n(x)$ is the probability of visiting $x$ at time
$n$ when starting at $o$. If there is some non-negative
number $h$ such that
$$
h=\lim_{n\to\infty} -\frac{1}{n}\E\bigl[\log \pi_n(X_n)\bigr],
$$
then $h$ is called the \emph{asymptotic entropy}, introduced in \cite{avez72}. For the rest of this section we
assume   that the transition probabilities are \emph{bounded away
from $0$}, that is, $p(x,y)\geq \varepsilon_0$ for some
$\varepsilon_0>0$ and all $x,y\in\T$. Under the assumption that
$(X_n)_{n\in\N_0}$ is transient we obtain the following theorem
that links the asymptotic entropy with the rate of escape.

\begin{thm}\label{th:entropy}
Assume that $Q$ is positive recurrent with invariant probability measure $\nu$
and that the
spectral radius of $(X_n)_{n\in\N_0}$ is strictly smaller than $1$. Let $\ell_0$ be the rate of escape w.r.t. the natural graph
metric. Then the entropy rate $h$ exists and satisfies
$$
h = \ell_0  \sum_{i,j\in G} -\nu(i)q(i,j)\, \log q(i,j)>0.
$$
\end{thm}
\begin{rem}
The assumptions of Theorem  \ref{th:entropy} are satisfied for transient
random walks on directed covers of finite graphs, see \cite[Sections 4 \&
5]{nagnibeda:02}. Thus, for the case of covers of finite graphs we get the
completely new
result that the
entropy exists and is strictly positive whenever the random walk
is transient. 
\end{rem}
\begin{rem}
The matrix $Q$ as defined by (\ref{def:q}) is the transition
matrix of the Markov chain $\bigl(\tau(X_{\e{k}})\bigr)_{k\in\N}$,
where $\e{k}$ is a random time as defined in (\ref{def:exittime}).
The sum on the right hand side of Theorem \ref{th:entropy} equals the 
entropy rate (for positive recurrent Markov chains) of
$\bigl(\tau(X_{\e{k}})\bigr)_{k\in\N}$ defined by 
$$
h_Q:=\lim_{n\to\infty} -\frac{1}{n}\log
\pi\bigl(\tau(X_{\e{1}}),\dots,\tau(X_{\e{n}})\bigr),
$$ 
where
$\pi(\tau_1,\dots,\tau_n)$ is the joint distribution of
$\bigl(\tau(X_{\e{1}}),\dots,\tau(X_{\e{n}})\bigr)$. That is,
$h=\ell\, h_Q$. For more details we refer to \cite[Chapter 4]{cover-thomas}.
\end{rem}
\begin{rem}
The proof of Theorem \ref{th:entropy} shows also that the entropy rate equals
the rate of escape with respect to the Green metric. Recall that the distance
of a vertex $x\in\T$ to $o$ w.r.t. the
Green metric is given by $-\log \P\bigl[\exists n\in\N_0: X_n=x \mid X_0=o\bigr]$. 
\end{rem}
A consequence of the proof of the last theorem is the following corollary that states that one can compute the entropy also individually:
\begin{cor}\label{cor:entropy2}
Under the assumptions of Theorem \ref{th:entropy},
$$
h=\liminf_{n\to\infty} -\frac{\log \pi_n(X_n)}{n} \textrm{ almost surely}.
$$
\end{cor}
Moreover, we get the following result:
\begin{thm}\label{thm:entropy3}
Under the assumptions of Theorem \ref{th:entropy}, 
$$
-\frac{1}{n}\log \pi_n(X_n) \xrightarrow{n\to\infty} h \ \textrm{ in } L_1,
$$
that is, $\int \bigl|-\frac{1}{n}\log \pi_n(X_n)-h\bigr| d\P \to 0$ for $n\to\infty$.
\end{thm}
Another  consequence of  Theorem \ref{th:entropy} follows with
\cite{kaimanovich98}, which gives an estimation (in some special
cases also an explicit formula) for the Hausdorff dimension of the
harmonic measure on the boundary of $\T$. Since almost every
random path on $\T$ converges to a point on the boundary $\partial
\T$ one can investigate the image $\varrho$ of the measure $\P$
under the mapping onto $\partial T$. Since we have a nearest neighbour random
walk this image is well-defined and it is called the \emph{harmonic measure} of $\P$ on
$\T$. For $\xi_1,\xi_2\in\partial \T$ let $\xi_1 \land \xi_2$ be
the confluent of the geodesics from $o$ to $\xi_1$ and from $o$ to
$\xi_2$. The Hausdorff dimension of $\varrho$
is defined to be
$$
\dim \varrho := \esssup_{\xi\in\partial\T} \liminf_{k\to\infty}
-\frac{\log \varrho\bigl(B_\xi^k \bigr)}{k},
$$
where $B_\xi^k:=\bigl\lbrace \zeta\in\partial \T \,\bigl|\, |\xi\land
\zeta|\geq k\bigr\rbrace$ for $k\in\N$, $\xi\in\partial\T$. The following
corollary follows directly with  \cite[Theorem 1.4.1 \& 1.5.3]{kaimanovich98}:

\begin{cor}\label{cor:entropy4}
Under the assumptions of Theorem \ref{th:entropy}, we have:
\begin{enumerate}
\item The Hausdorff dimension of $\varrho$ satisfies
$$
\sum_{i,j\in G} -\nu(i)q(i,j)\, \log q(i,j)
\leq \dim \varrho \leq \frac{-\log \eps_0}{\ell_0},
$$
where $\ell_0$ is the rate of escape w.r.t. the natural graph
metric. 
\item If the entropy $h$ also satisfies
$h=\lim_{n\to\infty} -\frac{1}{n}\log
  \pi_n(X_n)$ almost surely, then the Hausdorff
dimension of $\varrho$ is given by
$$
\dim \varrho = -\sum_{i,j\in G} \nu(i)q(i,j)\, \log q(i,j) = \frac{h}{\ell_0}.
$$
\end{enumerate}
\end{cor}

\begin{rem}
All the results in this section hold also for locally infinite
graphs with finite outdegrees.
\end{rem}

We conclude this section with two examples. While the first,
Example \ref{ex:entropy}, demonstrates  how  the asymptotic
entropy may be calculated explicitly in the finite case, the
second one, Example \ref{ex:speed1}, gives a sufficient condition for
positive recurrence of $Q$ in the infinite setting.

\begin{ex}\label{ex:entropy}
Consider a graph $G$ with vertex set $V=\{i_0,i_1,i_2\}$, edge set
$$
E=\bigl\lbrace (i_0,i_1), (i_0,i_2),(i_1,i_0),(i_2,i_1)\bigr\rbrace
$$
and its directed cover $\T$. We
define the following transition probabilities on $\T$:
$$
\begin{array}{c}
p(i_0,i_1):=\frac{1}{3},\ p(i_0,i_2):=\frac{1}{3},\ p(i_1,i_0):=\frac{1}{2},\ p(i_2,i_1):=\frac{3}{4},\\[1ex]
p(-i_0):=\frac{1}{3}, \ p(-i_1):=\frac{1}{2}, \ p(-i_2):=\frac{1}{4}.
\end{array}
$$

We can solve the system of polynomial equations
(\ref{f-equations}) with the help of \textsc{Mathematica} and hence
obtain a numerical approximation for the asymptotic entropy $h\approx 0.060499$.
%and for the Hausdorff dimension of the harmonic measure $\dim \varrho \approx 0.25614$.
\end{ex}

\begin{ex}\label{ex:speed1}
Suppose we are given a graph $G$ endowed with transition
probabilities  such that the random walk on $G$ is positive
recurrent with invariant probability measure $\nu_G$. Choose the
backward probabilities $p(-i)$ in such a way that the following
holds for every \mbox{$i,j\in G$:} if there are paths from $i_0$
to $i$ and from $i_0$ to $j$ in G of the same length, then
$p(-i)=p(-j)$. This condition implies that
$F\bigl(-\tau(y_1)\bigr)=F\bigl(-\tau(y_2)\bigr)$ if
$y_1^-=y_2^-$. Thus, the quotient
\mbox{$c_i:=\bigl(1-F(-j)\bigr)/\bigl(1-F(-i)\bigr)$} depends only
on $i$ if $p(i,j)>0$. This yields that
$$
1=\sum_{j\in G} q(i,j) = \sum_{j\in G} c_i p(i,j)\,
\overline{G}_i(1) = c_i \overline{G}_i(1) \bigl(1-p(-i)\bigr),
$$
or equivalently, $\overline{G}_i(1)=c_i^{-1}
\bigl(1-p(-i)\bigr)^{-1}$. But this implies that $\nu_G$ is also
the invariant probability measure of $Q$.
\par
If we have $p(-i)\leq 1/2-\eps$ for some $\eps>0$, then $G(o,o|z)$ has radius
of convergence strictly greater than $1$, providing that the rate of escape
w.r.t $|\cdot |$ exists and is strictly positive. Furthermore, the asymptotic
entropy exists and is strictly positive.
\end{ex}

\section{Proofs}\label{sec:proofs}

\subsection{Proof of Theorem \ref{thm:growth}}
The second statement of the theorem is just the following observation. Since $A$ has non-negative
entries we find
$$
\|A^n \|_\infty=\sup_{x\in\ell_\infty,\|x\|_\infty=1}
\|A^nx\|_\infty =\|A^n \1\|_\infty,
$$
and hence
 $|\T_n|\leq  \|A^n\|_\infty$.

The proof of the first statement is divided into two steps:

\noindent Step 1: $\lambda^+(A)=br(\T)$. This follows from the
results on homesick random walks: Corollary \ref{cor:homesick}
together with the fact that a homesick random walk on a  tree $\T$
is recurrent if $\lambda > br(\T)$ and transient if $\lambda < br (\T)$, compare with Theorem 3.5 in \cite{lyons:book}. For this
purpose, observe that $\lambda^+(M)=\lambda^+(A)/\lambda$.

\noindent Step 2: $br(\T)=\underline{gr}(\T)$.
It is convenient to introduce the notation
$$
\dim \sup
\partial \T:=\lim_{n\to\infty} \max_{v\in \T} \frac1n \log
|\T^v_n|
$$
and to follow the argumentation of \S 14.4 and \S 14.5
in \cite{lyons:book}. We need the following lemma:

\begin{lem}\label{lem:gr}
Let $\T$ be a tree of bounded geometry. Then there exists a
sequence of subtrees $\T^j=\T^{v_j}$ such that
\begin{equation}
\dim \partial \T^j \geq \left(1-\frac1j\right) \dim \sup \partial
\T.
\end{equation}
\end{lem}

\begin{proof}
We use the following fact, which is (14.19) in
\cite{lyons:book}.\newline 
\textbf{Fact:} For $j\geq 1$ there is
some unit flow $\theta_j$ on some subtree $\T^j=\T^{v_j}$ such that
$$
\frac1{|x|-|v_j|} \log \frac1{\theta_j(x)} \geq \left( 1-
\frac1j\right) \dim \sup \partial \T\quad\textrm{ for all } x\in \T^j.
$$
This
implies that for every ray $\xi=\langle
\xi_1,\xi_2,\ldots\rangle\in \partial \T^j$ we have
$$
\liminf_{n\to\infty} \frac1n \log \frac1{\theta_j(\xi_n)}\geq
\left(1-\frac1j\right) \dim \sup \partial \T.
$$
Recall the definition of the H\"older exponent of a unit flow (or
of its corresponding Borel probability measure on the
boundary respectively) $\mbox{H\"o}(\theta)(\xi):=\liminf_{n\to\infty} -\frac1n
\log \theta(\xi_n)^{-1}$. With this notation we have that
$\mbox{H\"o}(\theta_j)\geq \left(1-1/j\right) \dim \sup
\partial \T$. Those edges where the flow $\theta_j$ is positive
define a subtree $\U^j$ of $\T^j$. Theorem 14.15 of
\cite{lyons:book} implies that $\mbox{H\"o}(\theta_j)\leq \dim
\theta_j \leq \dim \partial \U^j$, and hence
$\mbox{H\"o}(\theta_j)\leq \dim \partial \T^j$.
\end{proof}

Due to Lemma \ref{lem:gr} there exists a sequence of subtrees
$\T^j=\T^{v_j}$ such that
\begin{equation*}
\dim \partial \T^j \geq \left(1-\frac1j\right) \dim \sup \partial
\T.
\end{equation*}
Furthermore, observe that $\dim \partial \T^j\leq \dim \partial
\T$ and $\log \underline{gr}(\T)\geq \dim\partial\T$; compare with Equation (\ref{eq:br_gr}). Hence,
\begin{eqnarray*}
\dim \partial \T & \geq & \left(1-\frac1j\right)\dim
\sup\partial\T\\
&\geq & \left(1-\frac1j\right)\liminf_{n\to\infty} \frac1n \log
|\T_n| \\
&=&
\left(1-\frac1j\right) \log \underline{gr}(\T)\\
&\geq & \left( 1-\frac1j\right) \dim\partial\T,
\end{eqnarray*}
Now, letting $j\to\infty$ yields $br(\T)=\underline{gr} \T$.
\begin{flushright}$\Box$\end{flushright}

\subsection{Proof of Theorem \ref{thm:br_mc} and \ref{thm:rec}}

\noindent\emph{Proof of Theorem \ref{thm:br_mc}.}
The idea is to couple a \emph{delayed} branching process $Z_t^*$ (in continuous
time) to the random walk $X_n$ on $\T$ (in discrete time) and to
show that the branching process dies out if and only if the random
walk visits the loop $(\orig, \orig)$. To this end, observe first
that the rates of the branching process $Z_t$ sum up to $1$, i.e.,
$p(-i)+\sum_{j\in G} p(i,j)=1$, and hence can be interpreted as
probabilities. The process $Z^*_t$ starts with one particle of
type $i_0$. With rate $p(-i_0)$ this particle dies. Observe
that the random walk started in $\orig$ returns to $\orig$ at the
first step with probability $p(-i_0)$. The particle produces an
offspring of type $j$ with rate $p(i_0,j)$. Observe that the
random walk on $\T$ is in a vertex $x\in\T$ with $|x|=1$ and label $j$ at time $1$ with
probability $p(i_0,j)$.  The \emph{delayed} (or \emph{sleepy})
process is defined inductively. As long as one particle of type
$i$ has one offspring alive (awake or sleeping) it is
\emph{sleeping}, i.e., it does neither die nor produce offspring.
If all its offspring have died it \emph{wakes up} and either dies
with rate $p(-i)$ or produces an offspring of type $j$ with
rate $p(i,j)$. If it does the latter it falls asleep again
and if it dies its direct ancestor wakes up.
One possibility to define $Z_t^*$ formally is through the following rates on the state space of non-backtracking paths (including the empty path $\left\langle \emptyset \right\rangle$) of $\T$. Let $x_0=\orig$ and

\begin{eqnarray*}
 \left\langle   x_0 \right\rangle &\rightarrow & \left\langle  x_0, x_1 \right\rangle \mbox{ at rate }  p(\tau(x_0),\tau(x_{1})), \mbox{ if } x_0=x_1^-,\\ 
\left\langle  x_0 \right\rangle &\rightarrow & \left\langle \emptyset \right\rangle \mbox{ at rate }   p(-\tau(x_0)), \\
 \left\langle \emptyset \right\rangle &\rightarrow & \left\langle \emptyset \right\rangle \mbox{ at rate }   1,
\end{eqnarray*}
and for $n\geq 1$
\begin{eqnarray*}
  \left\langle  x_0,\ldots, x_n \right\rangle &\rightarrow & \left\langle  x_0,\ldots, x_{n-1} \right\rangle \mbox{ at rate }   p(-\tau(x_n)), \\
  \left\langle   x_0,\ldots, x_n \right\rangle &\rightarrow & \left\langle  x_0,\ldots, x_n,x_{n+1} \right\rangle \mbox{ at rate }  p(\tau(x_n),\tau(x_{n+1})), \mbox{ if } x_n= x_{n+1}^-.
\end{eqnarray*}
Observe at this point that the path $\left\langle  x_0,\ldots, x_n \right\rangle$ corresponds to respectively one sleeping particle of type $\tau(x_i)$ ($i<n$) and one particle awake of type $\tau(x_n)$ in the genealogical order. The empty path $\left\langle \emptyset \right\rangle$ corresponds to the  extinction of the process. Let $S_n$ be the jump chain of $Z_t^*$ which is  the sequence of values taken by the continuous-time Markov chain
$Z_t^*.$  Define the projection $\phi$ from the space of paths to the set of vertices of $\T$ as $\phi(\left\langle  x_0,\ldots, x_n \right\rangle)=x_n$ and $\phi(\left\langle \emptyset \right\rangle)=\emptyset$.  Using standard arguments we can couple the two processes $S_n$ and $X_n$ such that $S_n=\emptyset$ if and only if $X_{n-1}=X_{n}=\orig.$ We conclude that $Z^*_t=\left\langle \emptyset \right\rangle$ for some $t>0$  if and only if the random walk $X_n$  visits the loop $(\orig, \orig)$.
It remains to prove that $Z^*_t=\left\langle \emptyset \right\rangle$ is equivalent to the extinction of the original process $Z_t$. Recall the interpretation of $Z^*_t$ as a delayed version of $Z_t.$ Hence, both processes can be seen as  functions on the same probability space and we can conclude with a standard coupling argument.
\begin{flushright}$\Box$\end{flushright}
\pagebreak[4]

\noindent\emph{Proof of Theorem \ref{thm:rec}.} Theorem
\ref{thm:rec}  is a  consequence of Theorem \ref{thm:br_mc} and
the following result \mbox{of \cite{bertacchi08}.} Consider a branching
random walk (BRW) as a continuous-time process where particles
live on a countable set $X$. Each particle lives on a site and,
independently of the others and the past history of the process,
has a exponential lifetime with mean $1$. A particle living at a
site $x$ gives birth to a new particle in $y$ with exponential
rate $k(x,y)$. Here $K=\bigl(k(x,y)\bigr)_{x,y\in X}$ is a matrix
with non-negative entries. In \cite{bertacchi08} it is shown
that there exists a critical value $\lambda_w$ depending on $K$ such that the
process dies out a.s. if $\lambda_w<1$ and  survives with positive
probability if $\lambda_w>1$. Furthermore, there is the following
characterization of this critical value, compare with Equation
(4.11) in \cite{bertacchi08}:
\begin{equation}
 \lambda_w=\lambda^+(K).
\end{equation}
The statement follows now with the observation that we can scale
our process by dividing each rate ($p(-i),~ p(i,j)$) at a vertex
$i$ by $p(-i)$ without influencing the survival of the process.
\begin{flushright}$\Box$\end{flushright}

\subsection{Prof of Theorem \ref{thm:rw_perco}}
We consider the infinite Galton-Watson process $Z_t$ with first
moments $M_\omega:=\lambda A_\omega$, compare with the paragraph
before Theorem \ref{thm:br_mc}. Due to Theorem \ref{thm:br_mc} it
suffices to prove the following:

\par\noindent
\textbf{Claim:} The process $Z_t$ survives (globally) if
and only if $\lambda > 1/2d$.

\par\noindent
First, observe that $\rho(A_\omega)=\limsup_{n\to\infty}
\left(A_\omega^{(n)}(i,j)\right)^{1/n}=2d$. This can be seen with
Equation (\ref{eq:spec_approx}) and the fact that $C(\omega)$
contains  balls of arbitrary large radius as subgraphs. There are
two types of survival for  infinite-type Galton Watson processes.
We say that the process survives globally if $Z_t>0$ for all $t$ and
survives locally if $Z_t(i)>0$ for infinitely many $t$ and all
($\Leftrightarrow$ some) $i$, compare with \cite{gantert:08}. Now,
Corollary 2.6 in \cite{mueller:ejp} implies that $Z_t$ survives
locally if and only if $\lambda>1/\rho(A_\omega)=1/(2d)$. Finally,
it remains to show that the process survives locally if and only
if it survives globally. But for $\lambda\leq 1/(2d)$, observe that
$\sum_y m_\omega^{(n)}(x,y)\leq 1$ for all $x$ and $n$. Hence, the
expected number of particles in generation $n$ is bounded by $1$.
Since $Z_t$ either converges to $0$ or $\infty$ we obtain that the
process does not survive globally if $\lambda\leq 1/(2d).$
\begin{flushright}$\Box$\end{flushright}

\subsection{Proof of Proposition \ref{prop:non_ergod}}

The non-ergodicity part in the supercritical case $\lambda^+(M)>1$
is obvious due to Theorem \ref{thm:rec}. Nevertheless, we give an
alternative proof that uses directly the definition of
$\lambda^+(M).$ This method works also for the critical case
$\lambda^+(M)=1$ when the supremum is attained. Furthermore, the
proof might be useful in order to understand the behaviour between
$\lambda^+(M)$ and $r_\infty(M)$.

\noindent\emph{Proof of Proposition \ref{prop:non_ergod}.}
The first steps are quite standard and use the tree
structure of our process, compare also with \cite{nagnibeda:02}.
The random walk is positive recurrent if and only if the
reversible (and stationary) measure $m$ defined in (\ref{eq:rev_m}) is finite, that is, if
$m(\T)=\sum_{x\in\T} m(x)<\infty$. For each $i\in  G$,
we construct a tree $\T_i$: it consists of the cone $\T^x$ of $\T$
with $x\in\T$ and $\tau(x)=i$ which is connected by a
(nondirected) single edge from $x$ to an additional vertex $\orig_i$. On $\T_i$ we consider the random walk with the same
transition probabilities as in $\T$ but where the probability from
$\orig_i$ to $x$ is $1$ and from $x$ to $\orig_i$ is $p(-i)$.
This random walk is obviously reversible with
reversible measure $m_i$, which is defined analogously to Equation
(\ref{eq:rev_m}). We can express $m(\T)$ in terms of the
measures $m_i$ of the  subtrees $\T_i$:
\begin{equation}\label{eq:mT}
m(\T)=\sum_{x\in\T} p(\orig,x) m_{\tau(x)}(\T_{\tau(x)}).
\end{equation}
We  approximate $\T_i$ with  finite subtrees $\T_i^n$ of height
$n$. The sequence $\bigl(m_i(\T_i^n)\bigr)_{n\in\N_0}$ is
increasing with limit $m_i(\T_i)$. We obtain the inductive
formula: $m_i(\T_i^0)=1$ and for each $i\in\N$
$$
m_i(\T_i^n)-1=\frac1{p(-i)} \biggl(1+ \sum_{j\in G} p(i,j)
  \bigl(m_j(\T_j^{n-1}) -1\bigr) \biggr).
$$
We can write the last equation in vector form using
$\m_n:=\bigl(m_i(\T_i^n)-1\bigr)_{i\in  G}$ and
$\mathbf{p}=\bigl(1/p(-i)\bigr)_{i\in G}$ as $\m_n=\p+M \m_{n-1}$. Thus,
$$
\m_n=\p+M\p+M^2\p+\cdots+ M^{n-1} \p
$$
and the sequence $(\m_n)_{n\in\N_0}$ will converge componentwise
to a finite limit if and only if $\sum_{n\geq 0} M^n\mathbf{p}$
converges in each component. If $\lambda^+(M)>1$ there is some
$\lambda^*>1$ and \mbox{$0<\mathbf{f}^*\in\ell_\infty$} such that
$M\mathbf{f}^*\geq \lambda^*\mathbf{f}^*$. Hence, there is some
$c>0$ such that \mbox{$\p=c\mathbf{f}^*+(\p-c\mathbf{f}^*)$} and
$(\p-c\mathbf{f}^*)>0$. At this point we need $p(-i)<1-\eps$
of Assumption (\ref{equ:p-i}). Finally,
$$
\sum_{n\geq 0} M^n \p \geq c \sum_{n\geq 0} (\lambda^*)^n
\mathbf{f}^*=\infty
$$ and non-ergodicity follows.
If $\lambda^+(M)=1$ and the supremum in the definition of
$\lambda^+(M)$ is attained, non-ergodicity follows analogously.

In order to show that $r_\infty:=r_\infty(M)<1$ implies ergodicity
we have to show that the sum $\sum_{n\geq 0} M^n \p<\infty$ is finite in every component if
\mbox{$r_\infty <1$.} Recall that $r_\infty=\lim_{n\to\infty}
\sqrt[n]{\|A^n\|_\infty}$. This provides that for every small $\eps>0$
there is some $N_\eps\in\N$ such that $1>r_\infty+\eps \geq
\sqrt[n]{\|A^n\|_\infty}$ for all $n\geq N_\eps$. Due to
Assumption (\ref{equ:p-i}) we get for every $i\in G$
and for $\eps$ small enough and $n\geq N_\eps$
 $$
\left(\sum_{n\geq 0} M^n \p\right) (i)
\leq \left\| \sum_{n\geq 0}
  M^n\p\right\|_\infty
\leq \sum_{n\geq 0}
\|M^n\p\|_\infty
\leq \sum_{n=0}^{N_\eps-1} \|M^np\|_\infty +
\|\p\|_{\infty}  \sum_{n\geq N_\eps} \bigl(r_\infty +\eps\bigr)^n,
$$ which is finite and independent of $i$.
\begin{flushright}$\Box$\end{flushright}

\subsection{Proof of Theorem \ref{thm:rateofescape}}

\begin{proof}[Proof of Theorem \ref{thm:rateofescape}.1]
Since the proof is an adaption of the arguments in
\cite{nagnibeda:02} we just give a  sketch.
 If $\Lambda<\infty$ then
$$
\frac{\e{k}}{k} \xrightarrow{k\to\infty} \Lambda=\sum_{i\in G}
  \nu(i)\,\frac{F'(-i|1)}{F(-i)} \quad \textrm{ almost surely.}
$$
With \cite{nagnibeda:02} we get $\ell_0:=\lim_{n\to\infty}
|X_n|/n= \lim_{k\to\infty} k/\e{k} = \Lambda^{-1}$ if
$\Lambda<\infty$.
\par
By assumption, the process $(\W{k})_{k\in\N_0}$ is positive recurrent with
invariant probability \mbox{measure $\nu$.} Thus, the process
$\bigl(\tau(\W{k-1}),\tau(\W{k})\bigr)_{k\in\N}$ has the invariant
probability measure $\nu_1(i,j)=\nu(i) q(i,j)$. An application of
the ergodic theorem for positive recurrent Markov chains yields
\begin{equation}\label{l0-convergence}
\frac{1}{n} \sum_{k=1}^n w\bigl(\tau(\W{k-1}),\tau(\W{k})\bigr) = \frac{l(\mathbf{W}_n)}{n}
\xrightarrow{n\to\infty} \int w(i,j)\,d\nu_1(i,j) = \sum_{i,j\in  G}
w(i,j) \nu_1(i,j).
\end{equation}
Observe that the sum on the right hand side is finite, since
$w(\cdot,\cdot)$ is bounded. In the case $\Lambda <\infty$ we
obtain, analogously to \cite{nagnibeda:02}, with $\mathbf{k}(n) :=
\max\{k\in\N_0| \e{k}\le n\}$
$$
\lim_{n\to\infty} \frac{l(X_n)}{n} = \lim_{n\to\infty}
 \frac{l(W_{\mathbf{k}(n)})}{\mathbf{k}(n)}\frac{\mathbf{k}(n)}{\mathbf{e}_{\mathbf{k}(n)}} \frac{\e{\mathbf{k}(n)}}{n}=
\ell_0  \lim_{k\to\infty} \frac{l(\W{k})}{k} \ \textrm{ almost
surely},
$$
since $\e{\mathbf{k}(n)}/n$ converges to $1$. The claim follows
now with (\ref{l0-convergence}).
\end{proof}

\begin{proof}[Proof of Theorem \ref{thm:rateofescape}.2]
Denote by $R$ the radius of convergence of $G(o,o|z)$. We have
$R>1$, since the spectral radius of $(X_n)_{n\in\N_0}$ is strictly
smaller than $1$ by assumption. Define for $x\in\T$ and $z\in\mathbb{C}$:
\begin{eqnarray*}
U(x,x|z)& := &\sum_{n\geq 1} \P\bigl[X_n=x,\forall m\in\{1,\dots,n-1\}: X_m\neq x
\,\bigl|\, X_0=x \bigr]z^n\\
&=& \sum_{j\in G}  p(i,j)\,z\,F(-j|z) +p(-i)\,z\, F(x^-,x|z).
\end{eqnarray*}

The proof splits up into the two following lemmas.

\begin{lem}\label{lem:g-estimate}
For $r\in [1,R)$ and all $x\in\T$ we have $G(x,x|r) \leq
1/(1-r/R)$.
\end{lem}
\begin{proof}
For every $x\in\T$ with $\tau(x)=i$, we have
\begin{equation}\label{g-u-equation}
\infty > G(x,x|r) =  \sum_{n\geq 0} U(x,x|r)^n =   \frac{1}{1-U(x,x|r)}\  \textrm{ for all } r\in[1,R),
\end{equation}
that is, $U(x,x|r)<1$ for all $r\in[1,R)$. Since $U(x,x|0)=0$ and
$U(x,x|z)$ is continuous, increasing and convex we have
$$
U(x,x|r) \leq \frac{r}{R},
$$
that is, $G(x,x|r) \leq 1/(1-r/R)$.
\end{proof}
\begin{lem}
There is a constant $C_F$ such that $F'(-i|1)\leq C_F$ for all $i\in G$.
\end{lem}
\begin{proof}
With Equation (\ref{g-u-equation}) we get
$$
U'(x,x|z) = \frac{G'(x,x|z)}{G^2(x,x|z)} \leq G'(x,x|z).
$$
Choose any $\eps\in(0,R-1)$ and define
$$
h(z) := \frac{1}{\eps} \frac{R(z-1)}{R-1-\eps}.
$$
We have $h(1)=0$ and $h(1+\eps)=1/(1-(1+\eps)/R)$. Since $G(x,x|z)$ is
increasing and convex in $[0,R)$ we get with Lemma \ref{lem:g-estimate} the
inequality $G'(x,x|1)\leq h'(1)$, and thus $U'(x,x|1)\leq h'(1)$.
\par
Let be $i\in G$ and choose any $x\in\T$ such that $p(\tau(x),i)>0$. Then:
$$
U(x,x|z) = p\bigl(-\tau(x)\bigr)z F(x^-,x|z) + \sum_{j\in G}
p\bigl(\tau(x),j\bigr) z F(-j|z).
$$
Differentiating yields
$$
U'(x,x|z) = p\bigl(-\tau(x)\bigr)\bigl(F(x^-,x|z)+ z F'(x^-,x|z)\bigr) + \sum_{j\in G}
p\bigl(\tau(x),j\bigr) \bigr(F(-j|z)+ z F'(-j|z)\bigr).
$$
Thus,
$$
U'(x,x|1)\geq p\bigl(\tau(x),i\bigr) F'(-i|1),
$$
or equivalently,
$$
F'(-i|1) \leq \eps_0^{-1} h'(1).
$$
\end{proof}

Finally,  $F(-i|1)\geq \eps_0$ together with the last lemma imply
that $\Lambda <\infty$ if $R>1$.
\end{proof}
\begin{proof}[Proof of Theorem \ref{thm:rateofescape}.3]
We now turn to the case $\Lambda=\infty$. Define for $N\in\N$ the function
$g_N:  G \times \N \to\N$ by $g_N(i,n):=n \land N$. Then obviously
\begin{equation}\label{gN-inequality}
\frac{1}{k}\sum_{l=1}^k g_N\bigl(\tau(\W{l}),\e{l}-\e{l-1}\bigr) \leq
\frac{\e{k}-\e{0}}{k} = \frac{\e{k}}{k}.
\end{equation}
This inequality holds for every $N\in\N$. The process
$\bigl(\tau(\W{l}),\e{l}-\e{l-1}\bigr)_{l\in\N}$ is also a positive recurrent
Markov chain;
compare with \cite{nagnibeda:02}. For each $N$ there is a constant $C_N$ such
that the left side of (\ref{gN-inequality})
converges to $C_N$ for almost every realisation of $(X_n)_{n\in\N_0}$. The
sequence $(C_N)_{N\in\N}$ is strictly increasing and diverges to $\infty$,
since $\Lambda=\infty$ and $g_N(i,n) \to g(i,n)$ for $N\to\infty$ with $g(i,n):=n$.
Thus, $\e{k}/k$ tends to infinity. This
yields $\liminf_{n\to\infty} |X_n|/n=0$ almost surely since
$$
0\leq \liminf_{n\to\infty} \frac{|X_n|}{n} \leq
\liminf_{n\to\infty}
\frac{|X_{\e{\mathbf{k}(n)}}|}{\e{\mathbf{k}(n)}} \leq
\liminf_{n\to\infty}
\frac{|X_{\e{\mathbf{k}(n)}}|}{\mathbf{k}(n)}\frac{\mathbf{k}(n)}{\e{\mathbf{k}(n)}}=0 \quad \textrm{ almost surely.}
$$
\end{proof}

\subsection{Proof of Theorems \ref{th:entropy} and Corollary
  \ref{cor:entropy2}}

\begin{proof}[Proof of Theorem \ref{th:entropy}]
Define for $x,y\in\T$ and $z\in\mathbb{C}$
\begin{eqnarray*}
L(x,y|z) & := &\sum_{n\geq 0} \P\bigl[X_n=y,\forall m\in \{1,\dots,n\}:
X_m\neq x \mid X_0=x\bigr]z^n.
\end{eqnarray*}
 If $y$ is a successor of $x$ in $\T$, then
\begin{equation}\label{l-equation}
L(x,y|z) = p(x,y)  z \overline{G}_{\tau(y)}(z).
\end{equation}
We have the following important equations, which follow by conditioning on the
last visit of $o$, the first visit of $x$ respectively:
\begin{equation}\label{gl-equations}
G(o,x|z) = G(o,o|z) L(o,x|z) = F(o,x|z) G(x,x|z).
\end{equation}
If $y\in\T$ lies on the unique geodesic from $x\in\T$ to $w\in\T$, then
$$
L(x,w|z) = L(x,y|z) L(y,w|z).
$$
Observe that the generating functions $L(\cdot,\cdot|z)$,
$\overline{G}_i(z)$, $G(\cdot,\cdot|z)$ have radii of convergence
of at least $R>1$, since the spectral radius of the random walk is
strictly smaller than $1$. Define for $x\in\T$
$$
l(x) := -\log L(o,x|1) = -\sum_{i=1}^{|x|} \log
L(x_{i-1},x_i|1),
$$ 
where $x_i$ is the unique element
on the geodesic from $o$ to $x$ at distance $i$ from $o$. This length function
arises from the weight function on the edges of $G$ defined by $w(i,j):=-\log L(x,y|1)$, where $x\in\T$ with
$\tau(x)=i$ and $y$ is a  successor of $x$ of type $j$. This weight function is
well-defined, since all subtrees $\T^{x_1}$ and $\T^{x_2}$ with
$\tau(x_1)=\tau(x_2)$ are isomorphic as
rooted trees. We claim
that the rate of escape w.r.t. the length function $l$ exists and
equals the asymptotic entropy. The technique of the proof which we
will give was motivated by \cite{benjamini-peres94}, where it is
shown that the asymptotic entropy of random walks on finitely generated groups
equals the rate of escape w.r.t. the Green metric.
\par
The proof of Theorem \ref{th:entropy} is  split  up into the
following lemmas.

\begin{lem}
$h:=\lim_{n\to\infty} l(X_n)/n$ exists and is non-negative.
\end{lem}
\begin{proof}
Observe that $L(x,y)\geq p(x,y)\geq \varepsilon_0$ whenever
$p(x,y)>0$. If $y\in\T$ is a successor of $x\in\T$, then we obtain
with Equation (\ref{l-equation}) and Lemma \ref{lem:g-estimate}
$$
L(x,y|1) = p(x,y) \,\overline{G}_{\tau(y)}(1) \leq p(x,y) / (1-1/R) \leq
(1-\varepsilon_0) / (1-1/R).
$$
Hence, the functions $L(x^-,x|1)$ are uniformly bounded in $x\in\T\setminus \{o\}$. Theorem
\ref{thm:rateofescape} provides that the rate of escape $h$ with
respect to $l$ exists. By
(\ref{gl-equations}), we get also
$$
h  =  \lim_{n\to\infty} -\frac{1}{n} \log F(o,X_n|1) - \frac{1}{n}
\log G(X_n,X_n|1) + \frac{1}{n} \log G(o,o|1).
$$
With Lemma \ref{lem:g-estimate} we have $1\leq G(x,x|1) \leq
1/(1-1/R)$ for every $x\in\T$. Since $F(o,X_n|1)\leq 1$ we get
$$
h=\lim_{n\to\infty} -\frac{1}{n}\log F(o,X_n|1) \geq 0.
$$
\end{proof}
By (\ref{gl-equations}), we can rewrite $h$ as

\begin{equation}\label{eq:h_in_termsof_G}
h=\lim_{n\to\infty} -\frac{1}{n} \log L(o,X_n|1) =
\lim_{n\to\infty} -\frac{1}{n} \log \frac{G(o,X_n|1)}{G(o,o|1)} =
\lim_{n\to\infty} -\frac{1}{n} \log G(o,X_n|1).
\end{equation}
Since
$$
G(o,X_n|1) = \sum_{m\geq 0} p^{(m)}(o,X_n) \geq p^{(n)}(o,X_n) = \pi_n(X_n),
$$
we have
\begin{equation}\label{equ:liminf-h}
\liminf_{n\to\infty} -\frac{1}{n} \log \pi_n(X_n) \geq h.
\end{equation}
The next aim is to prove $\limsup_{n\to\infty}
-\frac{1}{n}\E\bigl[\log \pi_n(X_n)\bigr] \leq h$. For this
purpose, we need the following lemma:
\begin{lem}
For every $r\in (1,R)$, $x\in\T$ and $m\in\N$ we have
$p^{(m)}(o,x)\leq G(o,x|r)  r^{-m}$.
\end{lem}
\begin{proof}
Denote by $C_r$ the circle with radius $r$ in the complex plane centered at
$0$. A straightforward computation shows that
$$
\frac{1}{2\pi i} \oint_{C_r} z^m \frac{dz}{z} = \delta_{m,0}.
$$
An application of Fubini's Theorem yields
\begin{eqnarray*}
\frac{1}{2\pi i} \oint_{C_r} G(o,x|z)\,z^{-m} \frac{dz}{z} & = &
\frac{1}{2\pi i} \oint_{C_r} \sum_{n\geq 0} p^{(n)}(o,x) z^n\,z^{-m}
\frac{dz}{z}\\
&=&
\frac{1}{2\pi i} \sum_{n\geq 0} p^{(n)}(o,x) \oint_{C_r} z^{n-m}
\frac{dz}{z} = p^{(m)}(o,x).
\end{eqnarray*}
Since $G(o,x|z)$ is analytic on $C_r$, $|G(o,x|z)|\leq G(o,x|r)$ for all
$|z|=r$. Thus,
$$
p^{(m)}(o,x) \leq \frac{1}{2\pi} r^{-m-1}  G(o,x|r) 2\pi r =
G(o,x|r)  r^{-m}.
$$
\end{proof}
Let be $x\in\T$ and let $x_i$ be the unique element on the geodesic from $o$ to
$x$ at distance $i$ \mbox{from $o$.} For $r<R$, iterated applications of equations
(\ref{l-equation}) and (\ref{gl-equations}) provide
\begin{eqnarray*}
G(o,x|r) & = & G(o,o|r) \,\prod_{i=1}^{|x|} L(x_{i-1},x_i|r)\\
& = & G(o,o|r) \,\prod_{i=1}^{|x|} p(x_{i-1},x_i)\,r\, \overline{G}_{\tau(x_{i})}(r)\\
&\leq & G(o,o|r) \bigl(r/(1-r/R)\bigr)^{|x|}.
\end{eqnarray*}
Thus,
\begin{equation}\label{eq:p_approx}
p^{(m)}(o,X_n) \leq G(o,o|r) \Bigl(\frac{r}{1-r/R}\Bigr)^n r^{-m}.
\end{equation}
We now need the following technical lemma:
\begin{lem}\label{cut-lemma}
Let $(A_n)_{n\in\N}$, $(a_n)_{n\in\N}$, $(b_n)_{n\in\N}$ be sequences of
positive numbers with $A_n=a_n+b_n$. Assume that $\lim_{n\to\infty}
-\frac{1}{n}\log A_n=c \in [0,\infty)$ and that $\lim_{n\to\infty} b_n/q^n
= 0$ for all $q\in(0,1)$. Then $\lim_{n\to\infty} -\frac{1}{n}\log a_n=c$.
\end{lem}
\begin{proof}Clearly, there is some $N\in\N$ such that $b_n<a_n$ for all $n\geq N$. We get for
all $n\geq N$:
\begin{eqnarray*}
-\frac{1}{n}\log(a_n+b_n) &\leq & -\frac{1}{n}\log(a_n) =
-\frac{1}{n}\log\Bigl(\frac{1}{2}a_n+\frac{1}{2}a_n\Bigr) \\
&\leq &
-\frac{1}{n}\log\Bigl(\frac{1}{2}a_n+\frac{1}{2}b_n\Bigr) \leq
 -\frac{1}{n}\log \left(\frac{1}{2}\right) -\frac{1}{n}\log(a_n+b_n).
\end{eqnarray*}
Letting $n\to\infty$ yields that $-\frac{1}{n}\log(a_n)$ tends to
$c$.
\end{proof}
In order to apply the last lemma let $A_n:=\sum_{m\geq 0}
p^{(m)}(o,X_n)$, $a_n:=\sum_{m=0}^{n^2-1} p^{(m)}(o,X_n)$ and
$b_n:=\sum_{m\geq n^2} p^{(m)}(o,X_n)$. Thus, for $r\in(1,R)$ we get with
(\ref{eq:p_approx})
$$
b_n\leq \sum_{m\geq n^2} G(o,o|r) \Bigl(\frac{r}{1-r/R}\Bigr)^n
\cdot r^{-m}= G(o,o|r)\Bigl(\frac{r}{1-r/R}\Bigr)^n
\frac{r^{-n^2}}{1-1/r}.
$$
Thus, $b_n$ decays faster than any geometric sequence. Lemma
\ref{cut-lemma} together with (\ref{eq:h_in_termsof_G}) yields
$$
h=\lim_{n\to\infty}  -\frac{1}{n} \log \sum_{m=0}^{n^2-1} p^{(m)}(o,X_n).
$$
\par
Since $G(o,X_n)\leq 1/(1-1/R)$ and $G(o,X_n)\geq F(o,X_n)\geq \eps_0^n$
we get by an application of the Dominated Convergence Theorem:
\begin{eqnarray*}
h & = &  \int \lim_{n\to\infty} -\frac{1}{n} \log \sum_{m=0}^{n^2-1}
p^{(m)}(o,X_n)\, d\P \\
&=&
\lim_{n\to\infty} \int -\frac{1}{n} \log \sum_{m=0}^{n^2-1}
p^{(m)}(o,X_n)\, d\P\\
&=& \lim_{n\to\infty} -\frac{1}{n} \sum_{x\in\T} p^{(n)}(o,x) \log \sum_{m=0}^{n^2-1}
p^{(m)}(o,x).
\end{eqnarray*}
Recall that Shannon's Inequality gives
$$
\sum_{x\in\T} p^{(n)}(o,x) \log \mu(x) \leq \sum_{x\in\T} p^{(n)}(o,x) \log p^{(n)}(o,x)
$$
for every finitely supported probability measure $\mu$ on $\T$.
Setting $\mu(x):=n^{-2} \sum_{m=0}^{n^2-1}
p^{(m)}(o,x)$ we get
\begin{eqnarray*}
h & \geq & \limsup_{n\to\infty} \left( -\frac{1}{n} \sum_{x\in\T}
p^{(n)}(o,x) \log n^2 -\frac{1}{n} \sum_{x\in\T} p^{(n)}(o,x) \log
p^{(n)}(o,x)\right) \\
&=& \limsup_{n\to\infty} -\frac{1}{n}\E\bigl[\log \pi_n(X_n)\bigr].
\end{eqnarray*}
Now we can conclude from Fatou's Lemma and (\ref{equ:liminf-h}) :
\begin{eqnarray}
h  \leq \int \liminf_{n\to\infty} \frac{-\log \pi_n(X_n)}{n} d\P & \leq &
\liminf_{n\to\infty} \int \frac{-\log \pi_n(X_n)}{n} d\P \nonumber\\
& \leq &
\limsup_{n\to\infty} -\frac{1}{n} \E\bigl[\log \pi_n(X_n)\bigr] \leq h.\label{eq:entropy-final}
\end{eqnarray}
Thus, $\lim_{n\to\infty} -\frac{1}{n} \E\bigl[\log
\pi_n(X_n)\bigr]$ exists and equals $h$.
It remains to show:
\begin{lem}
$$
h=\ell_0  \sum_{i,j\in G} -\nu(i)\,q(i,j) \log q(i,j).
$$
\end{lem}
\begin{proof}
For a moment let be $x\in\T$ with $|x|=n$ and let $x_j$ be the element on the geodesic from $o$ to
$x$ at distance $j$ from $o$. Then:
\begin{eqnarray}
l(x)
& = & -\log \prod_{j=1}^{n} L(x_{j-1},x_j|1)\nonumber
\\
&=&  -\log \prod_{j=1}^{n} p\bigl(\tau(x_{j-1}),\tau(x_j)\bigr) \,
\overline{G}_{\tau(x_{j-1})}(1) \frac{1-F\bigl(-\tau(x_{j})\bigr)}{1-F\bigl(-\tau(x_{j-1})\bigr)}
+\nonumber\\
&& \hspace{6cm} + \log
\frac{1-F\bigl(-\tau(x_n)\bigr)}{1-F(-i_0)}+\log
\frac{\overline{G}_{i_0}(1)}{\overline{G}_{\tau(x_n)}(1)}\nonumber\\
&=&  -\log \prod_{j=1}^{n} q\bigl(\tau(x_{j-1}),\tau(x_j)\bigr) +\log
\frac{1-F\bigl(-\tau(x_n)\bigr)}{1-F(-i_0)}+\log
\frac{\overline{G}_{i_0}(1)}{\overline{G}_{\tau(x_n)}(1)}.\label{q-entropy}
\end{eqnarray}
As $l(X_n)/n$ tends to $h$ almost surely, the subsequence
$\bigl(l(X_{\e{k}})/\e{k}\bigr)_{k\in\N}$ converges also to $h$.
Since $1\leq \overline{G}_i(1)\leq 1/(1-1/R)$ by Lemma
\ref{lem:g-estimate}, it follows with $x=X_{\e{k}}$ in
(\ref{q-entropy}) that
$$
\frac{1}{\e{k}} \log \frac{\overline{G}_0(1)}{\overline{G}_{\tau(X_{\e{k}})}(1)}
\xrightarrow{k\to\infty} 0 \quad \textrm{ almost surely}.
$$
By positive recurrence of $\bigl(\tau(X_{\e{k}})\bigr)_{k\in\N}$, an
application of the ergodic theorem yields
$$
-\frac{1}{k} \log \prod_{j=1}^{k}
q\bigl(\tau(X_{\e{j-1}}),\tau(X_{\e{j}})\bigr) \xrightarrow{n\to\infty}
h':=-\sum_{i,j\in G} \nu(i)\,q(i,j) \log q(i,j) \quad \textrm{ almost surely},
$$
whenever $h'<\infty$. In the latter case, since
$\lim_{k\to\infty}
k/\e{k}=\ell_0$ (see proof of Theorem \ref{thm:rateofescape}) and 
$\liminf_{n\to\infty} -\frac{1}{\e{k}}
\bigl(1-F(-\tau(X_{\e{k}}))\bigr)=0$ by ergodicity of
$\bigl(\tau(X_{\e{k}})\bigr)_{k\in\N}$, we
have
$$
h = \lim_{k\to\infty} \frac{l\bigl(X_{\e{k}}\bigr)}{\e{k}} =
\lim_{k\to\infty} \frac{l\bigl(X_{\e{k}}\bigr)}{k}\frac{k}{\e{k}}
= h'  \ell_0.
$$
In particular, $h>0$ since $\ell_0>0$ by Theorem
\ref{thm:rateofescape}.2.
\par
It remains to show that it cannot be that $h'=\infty$. For this
purpose, assume  $h'=\infty$. Let $N\in\N$ and define $h_N: G
\times G \to \R$ by $h_N(i,j):= N \land \bigl(-\log q(i,j)\bigr)$.
Then
$$
-\frac{1}{k} \sum_{j=1}^k \log h_N  \bigl(\tau(X_{\e{j-1}}),\tau(X_{\e{j}})\bigr) \xrightarrow{k\to\infty}
h_N':=-\sum_{i,j\in G} \nu(i)\,q(i,j) \log h_N(i,j) \quad \textrm{ almost surely}.
$$
Since $h_N(i,j)\leq -\log q(i,j)$ and $h'=\infty$ by assumption, there is for every
$M\in\R$ and almost every trajectory of $\bigl(\tau(X_{\e{k}})\bigr)_{k\in\N_0}$
an almost surely finite random time $\mathbf{T_q}\in\N$ such that for all $k\geq \mathbf{T_q}$
\begin{equation*}\label{M-inequality}
-\frac{1}{k} \sum_{j=1}^k \log q \bigl(\tau(X_{\e{j-1}}),\tau(X_{\e{j}})\bigr)
> M.
\end{equation*}
On the other hand there is for every $M>0$, every small $\eps >0$
and almost every trajectory an almost surely finite random time
$\mathbf{T_L}$ such that for all $k\geq \mathbf{T_L}$
\begin{eqnarray*}
&&-\frac{1}{\e{k}} \sum_{j=1}^k \log L\bigl(X_{\e{j-1}},X_{\e{j}}|1\bigr) \in (h-\eps,
h+\eps)\quad \textrm{ and}\\
&& -\frac{1}{\e{k}} \sum_{j=1}^k \log
q\bigl(\tau(X_{\e{j-1}}),\tau(X_{\e{j}})\bigr) =
-\frac{k}{\e{k}}\frac{1}{k} \sum_{j=1}^k \log
q\bigl(\tau(X_{\e{j-1}}),\tau(X_{\e{j}})\bigr)
> \ell_0 M-\eps.
\end{eqnarray*}
Furthermore, by positive recurrence of
$\bigl(\tau(X_{\e{k}})\bigr)_{k\in\N_0}$ there is an almost surely
finite random time $\mathbf{T} \geq \mathbf{T_L}$ such that
$$
-\frac{1}{\e{\mathbf{T}}}\log \frac{1-F(-\tau(X_{\e{\mathbf{T}}}))}{1-F(-0)}
\in (-\eps,\eps) \quad \textrm{ and } \quad
\frac{1}{\e{\mathbf{T}}} \log \frac{\overline{G}_0(1)}{\overline{G}_{\tau(X_{\e{\mathbf{T}}})}(1)}
\in (-\eps,\eps).
$$
Choose now $M>(h+4\eps)/\ell_0$. We obtain the desired
contradiction when we substitute in equality (\ref{q-entropy}) the
vertex $x$ by $X_{\e{\mathbf{T}}}$, divide by $\e{\mathbf{T}}$ on
both sides and see that the left side is in $(h-\eps,h+\eps)$ and
the rightmost side is larger than $h+\eps$.
\end{proof}
This finishes the proof of Theorem \ref{th:entropy}.
\end{proof}
\begin{proof}[Proof of Corollary \ref{cor:entropy2}]
Recall Inequality (\ref{equ:liminf-h}). Integrating both sides of this
inequality yields together with the inequality chain (\ref{eq:entropy-final})
that
$$
\int \liminf_{n\to\infty}  -\frac{\log \pi_n(X_n)}{n}  -h \,d\P = 0,
$$
providing that $h=\liminf_{n\to\infty} -\frac{1}{n} \log \pi_n(X_n)$ almost
surely.
\end{proof}

\subsection{Proof of Theorem \ref{thm:entropy3}}

To prove the theorem we need the following lemma:
\begin{lem}\label{lem:entropy3}
Under the assumptions of Theorem \ref{th:entropy}, 
$$
-\frac{1}{n}\log \pi_n(X_n) \xrightarrow{\P} h,
$$
that is, $-\frac{1}{n}\log \pi_n(X_n)$ converges in probability to the
asymptotic entropy.
\end{lem}
\begin{proof}
For every $\delta_1>0$ there is some index $N_{\delta_1}$ such that for all
$n\geq N_{\delta_1}$
$$
\int -\frac{1}{n} \log \pi_n(X_n)\,d\P \in (h-\delta_1,h+\delta_1).
$$
Furthermore, due to Corollary \ref{cor:entropy2} there is for every $\delta_2>0$ some index $N_{\delta_2}$ such
that for all $n\geq N_{\delta_2}$
\begin{equation}\label{equ:stochconv1}
\P\Bigl[-\frac{1}{n} \log \pi_n(X_n)>h-\delta_1\Bigr] > 1-\delta_2.
\end{equation}
Since $h=\lim_{n\to\infty} \int -\frac{1}{n}\log \pi_n(X_n)\,d\P$ it must be that
for every arbitrary but fixed $\eps>\delta_1$ and for $n$ large enough
$$
\P\Bigl[-\frac{1}{n}\log \pi_n(X_n)>h-\delta_1\Bigr]\cdot (h-\delta_1) + 
\P\Bigl[-\frac{1}{n}\log \pi_n(X_n)>h +\eps\Bigr]\cdot (\eps+\delta_1) 
\leq h+\delta_1,
$$
or equivalently,
$$
\P\Bigl[-\frac{1}{n}\log \pi_n(X_n)>h +\eps\Bigr] \leq
\frac{h+\delta_1 -\P\Bigl[-\frac{1}{n}\log \pi_n(X_n)>h-\delta_1\Bigr]\cdot (h-\delta_1)}{\eps+\delta_1}.
$$
If we let $\delta_2 \to 0$, we get
$$
\limsup_{n\to\infty} \P\Bigl[-\frac{1}{n}\log \pi_n(X_n)>h +\eps\Bigr] \leq \frac{2\delta_1}{\eps+\delta_1}.
$$
Since we can choose $\delta_1$ arbitrarily small we get
$$
\P\Bigl[-\frac{1}{n}\log \pi_n(X_n)>h +\eps\Bigr] \xrightarrow{n\to\infty} 0
\quad \textrm{ for all } \eps>0. 
$$
But this yields convergence in probability of $-\frac{1}{n}\log \pi_n(X_n)$ to $h$
together with (\ref{equ:stochconv1}).
\end{proof}
For any small $\eps>0$ and $n\in\N$, we define the events
$$
A_{n,\eps}:=\biggl[\Bigl|-\frac{1}{n}\log \pi_n(X_n)-h\Bigr|\leq \eps\biggr]
\ \textrm{ and } B_{n,\eps}:=\biggl[\Bigl|-\frac{1}{n}\log \pi_n(X_n)-h\Bigr|> \eps\biggr].
$$
There is some $N_\eps\in\N$ such that $\P[B_{n,\eps}]<\eps$ for all $n\geq
N_\eps$. Since $0\leq -\frac{1}{n}\log \pi_n(X_n)\leq \log \eps_0$ we can
conclude from Lemma \ref{lem:entropy3} for $n\geq N_\eps$:
\begin{eqnarray*}
&&\int \Bigl|-\frac{1}{n}\log \pi_n(X_n) -h\Bigr|\, d\P \\
& = &
\int_{A_{n,\eps}} \Bigl|-\frac{1}{n}\log \pi_n(X_n) -h\Bigr|\, d\P+
\int_{B_{n,\eps}} \Bigl|-\frac{1}{n}\log \pi_n(X_n) -h\Bigr|\, d\P \\[1ex]
&\leq &  \eps + \eps \log \eps_0 \xrightarrow{\eps\to 0} 0. 
\end{eqnarray*}
Thus, we have proved the theorem.
\begin{flushright}$\Box$\end{flushright}

\section*{Acknowledgement}
The authors are grateful to the German Research Foundation (DFG) for supporting
their projects, to which this work belongs. We acknowledge also the fruitful interplay and
exchange with the Austrian Science Fund (FWF) project FWF-P19115-N18 at Graz
University of Technology.

\begin{small}
\addcontentsline{toc}{chapter}{Bibliography}
\bibliography{bib}
\end{small}

\end{document}